\documentclass[a4paper, 10pt]{amsart}

\usepackage{amsmath, amsthm, amssymb, hyperref, color}
\usepackage{enumitem}

\newtheorem{theorem}{Theorem}[section]
\newtheorem{lemma}[theorem]{Lemma}
\newtheorem{corollary}[theorem]{Corollary}
\newtheorem{proposition}[theorem]{Proposition}

\theoremstyle{definition}
\newtheorem{example}[theorem]{Example}

\newcommand{\N}{\mathbb N}
\newcommand{\Z}{\mathbb Z}

\newcommand{\Q}{\mathbb Q}
\newcommand{\F}{\mathbb F}
\newcommand{\C}{\mathbb C}

 \DeclareMathOperator{\ord}{ord}
\DeclareMathOperator{\lcm}{lcm}

\DeclareMathOperator{\ch}{char}

\renewcommand{\t}{\, | \,}
\renewcommand{\time}{\negthinspace \times \negthinspace}

\newcommand{\be}{\begin{equation}}
\newcommand{\ee}{\end{equation}}

\newcommand{\bnml}{\begin{multline}}
\newcommand{\enml}{\end{multline}}
\newcommand{\buml}{\begin{multline*}}
\newcommand{\euml}{\end{multline*}}

\newcommand{\ber}{\begin{eqnarray}}
\newcommand{\eer}{\end{eqnarray}}

\numberwithin{equation}{section}

\begin{document}

\title{On clean, weakly clean, and feebly clean commutative group rings}

\author{Yuanlin Li }
\address{Department of Mathematics and Statistics,
Brock University, 1812 Sir Isaac Brock Way, St. Catharines, Ontario, Canada L2S 3A1
and
Faculty of Science, Jiangsu University, Zhenjiang, Jiangsu China
}
\email{yli@brocku.ca
}

\author{Qinghai Zhong}

\address{Institute for Mathematics and Scientific Computing,  University of Graz, NAWI Graz \\
 Heinrichstra{\ss}e 36, 8010 Graz, Austria, and 
School of Mathematics and statistics, Shandong University of Technology, Zibo, Shandong 255000, China.}
 \email{qinghai.zhong@uni-graz.at}

\thanks{This research was supported in part  by a  Discovery Grant from the Natural Sciences and Engineering Research Council of Canada (Grant No. RGPIN 2017-03903), by  the Austrian Science Fund FWF (Project Number P33499), and by National Natural Science Foundation of China (Grant No. 12001331).}



\begin{abstract}
A ring $R$ is said to be clean if each element of $R$ can be written as the sum of a unit and an idempotent. $R$ is said to be weakly clean if each element of $R$ is either a sum or a difference of a unit and an idempotent,  and $R$ is said to be feebly clean if every element $r$ can be written as $r=u+e_1-e_2$, where $u$ is a unit and $e_1,e_2$ are orthogonal idempotents. Clearly  clean rings are weakly clean rings and both of them are feebly clean.  In a recent article (J. Algebra Appl. 17 (2018), 1850111(5 pages)), McGoven characterized when the group ring  $\Z_{(p)}[C_q]$ is weakly  clean and feebly clean, where $p, q$ are distinct primes. In this paper, we consider a more general setting. Let $K$ be an algebraic number field, $\mathcal O_K$  its ring of integers,  $\mathfrak p\subset \mathcal O$ a nonzero prime ideal, and $\mathcal O_{\mathfrak p}$  the  localization of $\mathcal O$ at  $\mathfrak p$.  We investigate when the group ring $\mathcal O_{\mathfrak p}[G]$ is  weakly clean and feebly clean, where $G$ is a finite abelian group, and establish   an explicit characterization for such a group ring to be weakly clean and feebly clean  for the case when $K=\Q(\zeta_n)$ is  a cyclotomic field or $K=\Q(\sqrt{d})$ is  a quadratic field.
\end{abstract}

\maketitle

\bigskip
\section{Introduction}

Let $R$ be an associative ring with identity $1\neq 0$. The units and idempotents are key elements to determine the structure of the ring.
Following Nicholson \cite{Ni77}, a ring is called clean if every element is a sum of a unit and an idempotent. This class of rings includes commutative zero-dimensional rings, unit-regular rings, local rings, and semi-perfect rings. Since Camillo and Yu's publication \cite{Ca-Yu94} in 1994,
clean rings have captured wide attention in the literature. For some account of clean rings, the reader may consult \cite{Mc06b}.

It is easily seen that $R$ is clean if and only if every element is a difference of a unit and an idempotent. Thus it is natural to study the rings with the property that every element is either a sum or a difference of a unit and an idempotent.
These rings were first considered by Anderson and Camillo \cite{An-Ca02} and were called weakly clean rings by Ahn and Anderson \cite{Ah-An06}. For properties of weakly clean rings and connections with clean rings, we refer to \cite{Ah-An06,Ch-Qu11, Da14, Ko-Sa-Zh17}. Since several important
properties of clean rings  do not generalize to weakly clean rings, Arora and Kundu \cite{Ar-Ku17} introduced the family of feebly clean rings which has the property that every element $r$ can be written as $r=u+e_1-e_2$, where $u$ is a unit and $e_1,e_2$ are orthogonal idempotents. Then both clean rings and weakly clean rings are feebly clean.

Let $G$ be a group and let $C_n$ denote a cyclic group of order $n$. We denote by $R[G]$ the group ring of $G$ over $R$. The question of when a group ring $R[G]$ is clean was first considered by Han and Nicholson \cite{Ha-Ni01}. Then the cleanness of group rings has found wide
interest in the literature (for a sample out of
many see \cite{Ca-Kh-La-Ni-Zh06, Ch-Zh07}). In general, the question seems to be difficult. It is even unknown when $R[C_2]$ is clean. A special case of \cite[Corollary 20]{Zh10} is that
  if $R$ is a commutative  ring and $G$ is an  elementary abelian $2$-group, then $R[G]$ is clean if and only if $R$ is clean.

  Let $K$ be an algebraic number field, $\mathcal O$ its ring of integers, $\mathfrak p\subset \mathcal O$ a nonzero prime ideal, and   $\mathcal O_{\mathfrak p}$ denote the  localization of $\mathcal O$ at  $\mathfrak p$.
In \cite{Im-Mc14a}, Immormino and McGovern gave a characterization of when $\Z_{(p)}[G]$ is clean. Recently, we \cite{Li-Zh20} characterized  when  the group ring $\mathcal O_{\mathfrak p}[G]$ is clean, provided that $K$ is  a cyclotomic field or  a quadratic field.
Most recently, in \cite{Mc18}, McGovern gave a characterization of when $\Z_{(p)}[C_q]$ is weakly  clean and feebly clean, where $p, q$ are distinct primes.
In this paper, we consider a more general setting and establish a characterization of when the group ring $\mathcal O_{\mathfrak p}[G]$ is weakly clean and feebly clean, provided that $K$ is  a cyclotomic field or  a quadratic field.

We proceed as follows. In Section 2, we fix our notation and gather the required tools. In Section 3,
we first provide a complete characterization of when $\Z_{(p)}[G]$ is clean,  weakly clean and feebly clean (Theorem \ref{thm1}), 
which extends the above mentioned result of McGovern.   We  then establish an explicit
characterization of  when  $\mathcal O_{\mathfrak p}[G]$ is weakly clean and feebly clean, where $K$ is  a cyclotomic field (Theorem \ref{main1}) or  a quadratic field (Theorem \ref{main2}).


\bigskip
\section{Preliminaries}
\medskip

For a finite abelian group $G$, we denote by $\exp(G)$ the exponent of $G$. We denote by $\N$ the set of all positive integers and $\N_0=\N\cup\{0\}$.
For  $n\in \N$,  we denote by $\varphi(n)$ the Euler function.
Let $n\in \N$ and let $n=p_1^{k_1}\ldots p_s^{k_s}$ be the prime factorization, where $s, k_1,\ldots, k_s\in \N$ and $p_1,\ldots, p_s$ are pair-wise distinct primes. It is well-known that
 \begin{align*}
 &\varphi(n)=\prod_{i=1}^s\varphi(p_i^{k_i})=\prod_{i=1}^sp_i^{k_i-1}(p_i-1)\\
\text{ and }\quad &(\Z/n\Z)^{\times}\cong (\Z/p_1^{k_1}\Z)^{\times}\times\ldots\time(\Z/p_s^{k_s}\Z)^{\times}\,.
\end{align*}
Furthermore,
\begin{align*}
&(\Z/p_i^k\Z)^{\times}\cong C_{p_i^{k-1}(p_i-1)}& \quad &\text{ where } p_i\ge 3,\\
&(\Z/2^{\ell}\Z)^{\times}=\langle-1\rangle\times\langle5\rangle\cong C_2\oplus\C_{2^{\ell-2}} &\quad &\text{ where }\ell\ge 3,\\
\text{ and }\quad &(\Z/4\Z)^{\times}\cong C_2\,.&&
\end{align*}

For every $m\in \N$ with $\gcd(m,n)=1$, we denote by $\ord_nm=\ord_{(\Z/n\Z)^{\times}}m$ the multiplicative order of $m$ modulo $n$.
If $\ord_nm=\varphi(n)$, we say $m$ is a primitive root of $n$ and $n$ has a primitive root if and only if $n=2$, $4$, $q^{\ell}$, or $2q^{\ell}$, where $q$ is an odd prime and $\ell\in \N$. Let $n_1\in \N$ be another integer with $\gcd(n_1,m)=1$. Then
\begin{align*}
&\ord_nm\le \ord_{nn_1}m\le n_1\ord_nm\,,\\
\text{ and }\quad &\lcm(\ord_nm,\ord_{n_1}m)=\ord_{\lcm(n,n_1)}m\,.
\end{align*}

Let $\zeta_n$ be the $n$th primitive root of unity over $\Q$. Then $[\Q(\zeta_n):\Q]=\varphi(n)$. Let $m$ be another positive integer. Then
\begin{align*}
&\Q(\zeta_n)\cap \Q(\zeta_m)=\Q(\zeta_{\gcd(n,m)})\\
\text{ and }\quad  &\Q(\zeta_n)(\zeta_m)=\Q(\zeta_{\lcm(n,m)})\,.
\end{align*}

Let $R$ be a ring. We denote by $U(R)$ the group of units of $R$.

\begin{lemma}\label{lemm1}
Let $R$ be a commutative indecomposable ring. Then
\begin{enumerate}[label={(\arabic*)}, font={\bfseries}]
\item $R$ is clean if and only if $R$ is local.

\item $R$ is weakly clean but not clean if and only if  $R$  has exactly two maximal ideals and $2\in U(R)$.

\item $R$ is feebly clean  if and only if $R$ is weakly clean.

\item If $2\not\in U(R)$, then $R$ is weakly clean if and only if $R$ is clean.
\end{enumerate}

\end{lemma}

\begin{proof}
This follows from \cite[Corollary 1.4]{Ah-An06}.
\end{proof}

\begin{lemma}\label{lemm2}
 Let $\{R_{\alpha}\}$ be a family of commutative rings and let $R=\prod R_{\alpha}$. Then
 \begin{enumerate}[label={(\arabic*)}, font={\bfseries}]
 \item $R$ is clean if and only if each $R_{\alpha}$ is clean.

 \item $R$ is weakly  clean if and only if each $R_{\alpha}$ is weakly clean and at most one $R_{\alpha}$ is not clean.

 \item $R$ is feebly clean if and only if each $R_{\alpha}$ is feebly clean.
 \end{enumerate}
\end{lemma}
\begin{proof}
This follows from \cite[Theorem 1.7]{Ah-An06} and \cite[Theorem 2.1.2]{Ar-Ku17}.
\end{proof}


\medskip

Let $R$ be a commutative ring with identity and  let $G$ be a finite abelian group of exponent $\exp(G)$.  Suppose $R$ is indecomposable and $\exp(G)\in U(R)$. For every positive divisor $d$  of $\exp(G)$, we denote $\lambda(d)=\mu(d)\nu(d)$, where $\mu(d)$ is the number of the cyclic subgroups of $G$ of order $d$ and $\nu(d)$ is the number of the monic irreducible divisors of the cyclotomic polynomial $\Phi_d(x)$ over $R$.  The number $\lambda(d)$ is called the $d$-number of the group ring $R[G]$. Let $\phi_d(x)$  be a monic  irreducible divisor   of $\Phi_d(x)$. Then there exists a root $\epsilon_d$ of $\phi_d (x)$ in some extension of $R$ such that $\phi_d(x)$ is a minimal polynomial of $\epsilon_d$  over $R$. Therefore $R[\epsilon_d]\cong R[x]/(\phi_d(x))$.


\begin{theorem}[{\cite[Theorem 4]{Mo-Na06}}]\label{t1}
	Let $G$ be a finite abelian group with exponent $n$ and let $R$ be a commutative ring  which is a direct product of $m$ indecomposable rings $R_i$ (e.g., $R$ is  noetherian).
	\begin{enumerate}[label={(\arabic*)}, font={\bfseries}]
		\item 	\[
		R[G]\cong \sum_{i=1}^mR_i[G]\,.
		\]
		
		\item If $n\in U(R)$, then for every $i\in \{1,2,\ldots,m\}$, we have
		$$R_i[G]\cong \sum_{d\mid n}\lambda_i(d)(R_i[x]/(\phi_{i,d}))\,,$$
		where $\lambda_i(d)$ is the $d$-number of the group ring $R_i[G]$ and $\phi_{i,d}$ is  a monic  irreducible divisor   of $\Phi_d(x)$ over $R_i$.
	\end{enumerate}

\end{theorem}

\begin{proposition}
	Let $G$ be a finite abelian group with exponent $n$ and let $R$ be a commutative ring with $n\in U(R)$ which is a direct product of $m$ indecomposable rings $R_i$  (e.g., $R$ is  noetherian).
	\begin{enumerate}[label={(\arabic*)}, font={\bfseries}]
		\item $R[G]$ is clean if and only if $R[C_n]$ is clean.
		
		\item $R[G]$ is feebly clean if and only if $R[C_n]$ is feebly clean.
		
	\end{enumerate}	
\end{proposition}
\begin{proof}

By Theorem \ref{t1}, we obtain 
	\[
		R[G]\cong \sum_{i=1}^mR_i[G]\, \cong \sum_{i=1}^m\sum_{d\mid n}\lambda_i(d)(R_i[x]/(\phi_{i,d}))\,
		\]  and 
\[
		R[C_n]\cong \sum_{i=1}^mR_i[C_n]\, \cong \sum_{i=1}^m\sum_{d\mid n} \nu_i(d)(R_i[x]/(\phi_{i,d}))\,
		\] where, $\lambda_i(d)=\mu(d)\nu_i(d)$ is the $d$-number of the group ring $R_i[G]$,  and $\mu(d)$ is the number of the cyclic subgroups of $G$ of order $d$, and $\nu_i(d)$ is the number of the monic irreducible divisors of the cyclotomic polynomial $\Phi_d(x)$ over $R_i$.
Now Statements (1) and (2) follow immediately from  Lemma \ref{lemm2}.	

\end{proof}

Suppose $(R, \mathfrak m)$ is  a commutative local ring and we denote $\overline{R}=R/\mathfrak m$. Then $\overline{R}$ is a field and we denote by $\ch \overline{R}$ the characteristic of $\overline{R}$. For any polynomial $f(x)=a_nx^n+\ldots+a_0\in R[x]$, we denote $\overline{f(x)}=\overline{a_n}x^n+\ldots+\overline{a_0}\in \overline{R}[x]$, where $\overline{a_i}=a_i+\mathfrak m$ for all $i\in \{0,\ldots,n\}$.

Let $\F_q$ be a finite field, where $q$ is a power of some prime $p$ and let $\overline{\zeta_n}$ be the $n$th primitive root of unity over $\F_p$ with $\gcd(n,q)=1$. Then $[\F_q(\overline{\zeta_n}):\F_q]=\ord_nq$.
Let $F$ be an arbitrary field and let $f(x)$ be a polynomial of $F[x]$. If $\alpha$ is a root of $f(x)$, then $[F(\alpha):F]=\deg(f(x))$ if and only if $f(x)$ is irreducible in $F[x]$.

\begin{proposition}\label{2.5}
	Let $(R, \mathfrak m)$ be a local domain with $\overline{R}$  a finite field, let $G$ be a finite abelian group with $n=\exp(G)\in U(R)$, and let $\phi_d(x)$ be a monic irreducible divisor of $\Phi_d(x)$ over $R$ for all divisors $d$ of $n$.
		\begin{enumerate}[label={(\arabic*)}, font={\bfseries}]
		\item Let $d$ be a divisor of $n$. Then $R[\zeta_d]$ is clean if and only if $\deg(\phi_d(x))=\ord_d(|\overline{R}|)$. Furthermore, $R[G]$ is clean if and only if  $\deg(\phi_d(x))=\ord_d(|\overline{R}|)$ for all divisors $d$ of $n$.
		
		\item Suppose $2\in U(R)$. Let $d$ be a divisor of $n$. Then $R[\zeta_d]$ is feebly clean but not clean if and only if $\deg(\phi_d(x))=2\ord_d(|\overline{R}|)$. Furthermore, $R[G]$ is feebly clean  if and only if  $\deg(\phi_d(x))\le 2\ord_d(|\overline{R}|)$ for all divisors $d$ of $n$.

		\item Suppose $2\not\in U(R)$. Let $d$ be a divisor of $n$. Then $R[\zeta_d]$ is feebly clean  if and only if $R[\zeta_d]$ is weakly clean if and only if $R[\zeta_d]$ is  clean. Furthermore, $R[G]$ is feebly clean  if and only if  $R[G]$ is weakly clean if and only if $R[G]$ is  clean.
		\end{enumerate}

\end{proposition}
\begin{proof}
		
		Since $R$ is a local domain, we obtain $R$ and $R[\zeta_d]$ are  indecomposable for all divisors $d$ of $n$. Let $d$ be a divisor of $n$.
		
		Note that $$\mathfrak m=J(R)\subset J(R[x])\subset \bigcap_{I \text{ is a maximal ideal with }\phi_d(x)\in I }I\,.$$
		Then $(\mathfrak m, \phi_d(x))$ is contained in every maximal ideal of $R[x]$ containing $\phi_d(x)$.
		Since
		\begin{align*}
		&\big(R[x]/(\phi_d(x))\big)/\big((\mathfrak m, \phi_d(x))/(\phi_d(x))\big)\cong R[x]/(\mathfrak m, \phi_d(x))\\
		\cong &\big(R[x]/\mathfrak mR[x]\big)/\big((\mathfrak m, \phi_d(x))/\mathfrak mR[x]\big)\cong R[x]/\mathfrak m[x]/(\overline{(\phi_d(x))})\cong\overline{R}[x]/(\overline{(\phi_d(x))})\,,
		\end{align*}
	 the number of maximal ideals of $R[x]/(\phi_d(x))$ equals the number of maximal ideals of $\overline{R}[x]/(\overline{\phi_d(x)})$.
		Since $d\in U(R)$ (as $n\in U(R)$) and $\overline{R}[x]$ is a PID,  the number of maximal ideals of $\overline{R}[x]/(\overline{\phi_d(x)})$ equals the number of  irreducible divisors of $\overline{\phi_d(x)}$ over $\overline{R}$, whence equals $$\frac{\deg(\phi_d(x))}{[\overline{R}(\overline{\zeta_d}):\overline{R}]}=\frac{\deg(\phi_d(x))}{\ord_{d}(|\overline{R}|)}\,.$$

		\smallskip
		\textbf{(1)} By Lemma \ref{lemm1}.1, we have $R[\zeta_d]$ is clean if and only if $R[\zeta_d]\cong R[x]/(\phi_d(x))$ is local if and only if  $\frac{\deg(\phi_d(x))}{\ord_{d}(|\overline{R}|)}=1$, i.e., $\deg(\phi_d(x))=\ord_d(|\overline{R}|)$.
		
		Furthermore, by Theorem \ref{t1}.2 and Lemma \ref{lemm2}, we have $R[G]$ is clean if and only if $R[\zeta_d]$ is clean for all divisors $d$ of $n$ if and only if  $\deg(\phi_d(x))=\ord_d(|\overline{R}|)$ for all divisors $d$ of $n$.

		\smallskip
		\textbf{(2)} Suppose $2\in U(R)$. By Lemma \ref{lemm1}, we have $R[\zeta_d]$ is feebly clean but not clean  if and only if $R[\zeta_d]\cong R[x]/(\phi_d(x))$ has exactly two maximal ideals if and only if   $\frac{\deg(\phi_d(x))}{\ord_{d}(|\overline{R}|)}=2$, i.e., $\deg(\phi_d(x))=2\ord_d(|\overline{R}|)$.
		
		Furthermore, by Theorem \ref{t1}.2 and Lemma \ref{lemm2}, we have $R[G]$ is feebly clean  if and only if $R[\zeta_d]$ is feebly clean for all divisors $d$ of $n$  if and only if $\deg(\phi_d(x))\le 2\ord_d(|\overline{R}|)$ for all divisors $d$ of $n$.
		
		\smallskip
		\textbf{(3)} Suppose $2\not\in U(R)$. Then $2\not\in U(R[\zeta_d])$. Note that $R[\zeta_d]$ is indecomposable. It follows
		from Lemma \ref{lemm1} that $R[\zeta_d]$ is feebly clean  if and only if $R[\zeta_d]$ is weakly clean if and only if $R[\zeta_d]$ is  clean.
		
		Furthermore, by Theorem \ref{t1}.2 and Lemma \ref{lemm2}, we have $R[G]$ is feebly clean  if and only if $R[\zeta_d]$ is feebly clean for all divisors  $d$ of $n$  if and only if $R[\zeta_d]$ is  clean for all divisors $d$ of $n$ if and only if $R[G]$ is  clean. The assertion follows.
	\end{proof}

Next we recall some facts from  elementary number theory which will be used in the proofs of the following propositions without further mention.   Let $n=q^r>2$ be a prime power and let $p$ be a prime with $p\neq q$. If $\ord_np<\ord_{q^{r+1}p}$, then $\ord_{q^{r+s}}p=q^s\ord_{q^r}p$ for every $s\in \N$. In particular, if $q$ is odd and $q\t \ord_np$, then $\ord_{q^{r+s}}p=q^s\ord_{q^r}p$ for every $s\in \N$; if $q=2$ and $4\t \ord_np$, then $\ord_{2^{r+s}}p=2^s\ord_{2^r}p$ for every $s\in \N$.

\begin{proposition}\label{p2.6}
	Let $n$ be a positive integer and let $p$ be an odd  prime with $p\nmid n$. For $r\in \N$, we set
	$\epsilon(r)=\left\{
	\begin{aligned}
	1, &\text{ if }r=1\,,\\
	2, &\text{ if }r\ge 2\,.
	\end{aligned}\right.$
	\begin{enumerate}[label={(\arabic*)}, font={\bfseries}]
		\item $\varphi(n)=\ord_np$ if and only if one of the following holds
		\begin{enumerate}
			\item $n=1$ or $2$
			
			\item $n=4$ and $p\equiv 3\pmod 4$.
			

			\item $n=q^r$ or $2q^r$, where $q$ is an odd prime and $r\in\N$,  and  $p$ is a primitive root of $q^{\epsilon(r)}$.
		\end{enumerate}
		
		\item $\varphi(n)=2\ord_np$ if and only if one of the following holds
		\begin{enumerate}
			\item $n=4$ and $p\equiv 1\pmod 4$; or $n=8$ and $p\not\equiv 1\pmod 8$.

			\item $n=2^r$ with $r\ge4$ and $p\equiv \pm3,\pm5\pmod {16}$.
			
			
			\item $n=q^r$ or $2q^r$, where $q$ is an odd prime and $r\in \N$,  and  $\ord_{q^{\epsilon(r)}}p=q^{\epsilon(r)-1}(q-1)/2$.


			\item $n=4q^r$, where $q$ is an odd prime and $r\in \N$,  and  either  $p$ is a primitive root of $q^{\epsilon(r)}$, or $p\equiv 3\pmod 4$, $q\equiv 3\pmod 4$, and $\ord_{q^{\epsilon(r)}}p=q^{\epsilon(r)-1}(q-1)/2$.

			
				
				

			\item $n=q_1^{r_1}q_2^{r_2}$ or $2q_1^{r_1}q_2^{r_2}$, where $r_1,r_2\in \N$, $q_1$ and $ q_2$ are distinct odd primes with $q_2\equiv 3\pmod 4$ and $\gcd(q_1^{\epsilon(r_1)-1}(q_1-1), q_2^{\epsilon(r_2)-1}(q_2-1))=2$,  and either $p$ is a primitive root of both $q_1^{\epsilon(r_1)}$ and $q_2^{\epsilon(r_2)}$, or $p$ is a primitive root of $q_1^{\epsilon(r_1)}$ and $\ord_{q_2^{\epsilon(r_2)}}p=q_2^{\epsilon(r_2)-1}(q_2-1)/2$.
		\end{enumerate}
	
	\item $\varphi(n)=4\ord_np$ if and only if one of the following holds
	\begin{enumerate}	
		
		\item $n=8$ and $p\equiv 1\pmod 8$; or $n=16$ and $p\equiv 7,9,15\pmod {16}$.
				
		\item $n=2^r$ with $r\ge 5$ and $p\equiv \pm7,\pm9\pmod {32}$.
		
		
		\item $n=q^r$ or $2q^r$, where  $r\in \N$, $q$ is an odd prime with $q\equiv 1\pmod 4$,  and  $\ord_{q^{\epsilon(r)}}p=q^{\epsilon(r)-1}(q-1)/4$.
		
		\item $n=q_1^{r_1}q_2^{r_2}$ or $2q_1^{r_1}q_2^{r_2}$, where $r_1,r_2\in \N$, $q_1,q_2$ are distinct  odd primes,  and  one of the following holds.
		\begin{itemize}
			\item $p$ is a primitive root of both $q_1^{\epsilon(r_1)}$ and $ q_2^{\epsilon(r_2)}$  and $\gcd(q_1^{\epsilon(r_1)-1}(q_1-1), q_2^{\epsilon(r_2)-1}(q_2-1))=4$.
			
			\item $p$ is a primitive root of $q_1^{\epsilon(r_1)}$, $\ord_{q_2^{\epsilon(r_2)}}p=q_2^{\epsilon(r_2)-1}(q_2-1)/2$,  and $\gcd(q_1^{\epsilon(r_1)-1}(q_1-1), q_2^{\epsilon(r_2)-1}(q_2-1)/2)=2$.

			\item $\ord_{q_1^{\epsilon(r_1)}}p=q_1^{\epsilon(r_1)-1}(q_1-1)/2$, $\ord_{q_2^{\epsilon(r_2)}}p=q_2^{\epsilon(r_2)-1}(q_2-1)/2$, and $\gcd(q_1^{\epsilon(r_1)-1}(q_1-1)/2, q_2^{\epsilon(r_2)-1}(q_2-1)/2)=1$.
			
			\item $p$ is a primitive root of $q_1^{\epsilon(r_1)}$, $\ord_{q_2^{\epsilon(r_2)}}p=q_2^{\epsilon(r_2)-1}(q_2-1)/4$,  and $\gcd(q_1^{\epsilon(r_1)-1}(q_1-1), q_2^{\epsilon(r_2)-1}(q_2-1)/4)=1$.			
		\end{itemize}

			\item $n=q_1^{r_1}q_2^{r_2}q_3^{r_3}$ or $2q_1^{r_1}q_2^{r_2}q_3^{r_3}$, where $r_1,r_2,r_3\in\N$, $q_1,q_2, q_3$ are distinct  odd primes,  and  one of the following holds.
		\begin{itemize}
			\item $p$ is a primitive root of  $q_1^{r_1}, q_2^{r_2}, q_3^{r_3}$  and $\gcd(q_i^{r_i-1}(q_i-1), q_j^{r_j-1}(q_j-1))=2$  for all distinct $i,j\in \{1,2,3\}$.
			
			\item $p$ is a primitive root of $q_1^{r_1}, q_2^{r_2}$, $\ord_{q_3^{r_3}}p=q_3^{r_3-1}(q_3-1)/2$ is odd,  and $\gcd(q_i^{r_i-1}(q_i-1), q_j^{r_j-1}(q_j-1))=2$ for all distinct $i,j\in \{1,2,3\}$.

			\item $p$ is a primitive root of $q_1^{r_1}$, $\ord_{q_2^{r_2}}p=q_2^{r_2-1}(q_2-1)/2$ is odd, $\ord_{q_3^{r_3}}p=q_3^{r_3-1}(q_3-1)/2$ is odd, and $\gcd(q_i^{r_i-1}(q_i-1), q_j^{r_j-1}(q_j-1))=2$ for all distinct $i,j\in \{1,2,3\}$.

%
		\end{itemize}
		
		\item $n=4q^r$, where $r\in \N$, $q$ is a prime,  and one of the following holds
		
		\begin{itemize}
			\item $p\equiv 1\pmod 4$ and $\ord_{q^r}p=q^{r-1}(q-1)/2$;
			
			\item $p\equiv 3\pmod 4$, $q\equiv 1\pmod 4$,  and $\ord_{q^r}p=q^{r-1}(q-1)/2$;
			
			\item $p\equiv 3\pmod 4$, $q\equiv 5\pmod 8$, and $\ord_{q^r}p=q^{r-1}(q-1)/4$.
		\end{itemize}
	
	\item $n=4q_1^{r_1}q_2^{r_2}$, where $r_1,r_2\in \N$, $q_1$ and $q_2$ are distinct  odd primes with $q_2\equiv 3\pmod 4$,  and one of the following holds
	
	\begin{itemize}
		
		\item  $\gcd(q_1^{\epsilon(r_1)-1}(q_1-1), q_2^{\epsilon(r_2)-1}(q_2-1))=2$,  and  $p$ is a primitive root of both $q_1^{\epsilon(r_1)},q_2^{\epsilon(r_2)}$.
		
		\item   $\gcd(q_1^{\epsilon(r_1)-1}(q_1-1), q_2^{\epsilon(r_2)-1}(q_2-1))=2$,  $p$ is a primitive root of $q_1^{\epsilon(r_1)}$, and $\ord_{q_2^{\epsilon(r_2)}}p=q_2^{\epsilon(r_2)-1}(q_2-1)/2$.
				
		\item $p\equiv 3\pmod 4$, $q_1\equiv 3\pmod 4$,  $\ord_{q_1^{r_1}}p=q_1^{r_1-1}(q_1-1)/2$, $\ord_{q_2^{r_2}}p=q_2^{r_2-1}(q_2-1)/2$, and $\gcd(q_1^{r_1-1}(q_1-1)/2, q_2^{r_2-1}(q_2-1)/2)=1$.
	\end{itemize}

		\item $n=8q^r$, where $r\in \N$, $q$ is an odd prime,  and  one of the following holds
		\begin{itemize}
			\item  $p$ is a primitive root of $q^r$;
			
			\item $p\not\equiv 1\pmod 8$, $q\equiv 3\pmod 4$, and $\ord_{q^r}p=q^{r-1}(q-1)/2$.		
		\end{itemize}

		\item $n=2^tq^r$, where $t\ge 4$ and  $q$ is an odd prime with $q\equiv 3\pmod 4$, $p\equiv \pm3,\pm5\pmod {16}$, and  either $p$ is a primitive root of $q^r$, or  $\ord_{q^r}p=q^{r-1}(q-1)/2$.

	\end{enumerate}
	\end{enumerate}
\end{proposition}

\begin{proof}
Note that if $n\equiv 2\pmod 4$, then $\ord_np=\ord_{n/2}p$ and $\varphi(n)=\varphi(n/2)$, whence $\varphi(n)=k\ord_np$ if and only if $\varphi(n/2)=k\ord_{n/2}p$, where $k\in \{1, 2, 4\}$. It suffices to consider the case that $n\not\equiv 2\pmod 4$.	Now suppose $n\not\equiv 2\pmod 4$. In particular, if $2\mid n$, then $4\mid n$; and if $2\nmid n$ and the assertion holds for $n$, then the assertion also holds for $2n$.

\medskip
\textbf{(1)} The assertion is clear.

\medskip

\textbf{(2)} We distinguish five cases depending on the prime divisors of $n$.

\smallskip
 {\bf Case 1}. $n=2^r$ with $r\ge 2$.
\smallskip

Suppose $\varphi(n)=2\ord_np$.
 If $n=4$, then $\ord_4p=1$ and hence  $p\equiv 1\pmod 4$. If $n=8$, then $\ord_8p=2$ and hence  $p\not\equiv 1\pmod 8$. It follows that (a) holds.

If $r\ge 4$, then $\ord_{2^r}p=\varphi(2^r)/2=2^{r-2}$, whence $p^{4}\equiv 1\pmod {16}$ but $p^2\not\equiv 1\pmod {16}$. Therefore $p\not\equiv \pm1, \pm7\pmod {16}$ and hence (b) holds.

Suppose (a) holds. It is easy to see that $\ord_np=\varphi(n)/2$. Suppose (b) holds. Then
$p^2\not\equiv 1\pmod {16}$ and $p^4\equiv 1\pmod {16}$, whence $\ord_{2^4}p=2^{4-2}$ and hence $\ord_{2^r}p=2^{r-2}=\varphi(2^r)/2$.

\smallskip
{\bf Case 2}. $n=q^r$ with $q$ an odd prime and $r\in \N$.
\smallskip

Then $\ord_{n}p=\varphi(n)/2$ if and only if  (c)  holds.

\smallskip
{\bf Case 3}. $n=2^tm$ where $m\ge 3$ is odd and $t\ge 2$.
\smallskip

Suppose $\varphi(n)=2\ord_np$.
Then $\varphi(n)/2=2^{t-2}\varphi(m)=\ord_{n}p=\lcm(\ord_{2^t}p, \ord_{m}p)$. If $t\ge 3$, then $\ord_{2^t}p$ divides $2^{t-2}$.  Since $\varphi(m)$ is even, it follows  that  $$2^{t-2}\varphi(m)=\lcm(\ord_{2^t}p, \ord_{m}p)\le \lcm(2^{t-2}, \varphi(m))\le 2^{t-3}\varphi(m)\,,$$ a contradiction. Thus $n=4m$ and, since $\varphi(m)$ is even, we have $$\varphi(m)= \lcm(\ord_{4}p, \ord_{m}p)\le \lcm(2, \varphi(m))=\varphi(m)\,.$$
Then either $\ord_mp=\varphi(m)$ or $p\equiv 3\pmod 4$ and $\ord_mp=\varphi(m)/2$ is odd.
Since $m$ is odd, both cases imply that $m=q^r$ for some odd prime $q$.
If $\ord_mp=\varphi(m)$, then $p$ is a primitive root of $q^{\epsilon(r)}$. If $p\equiv 3\pmod 4$ and $\ord_mp=\varphi(m)/2$ is odd, then $\ord_{q^r}p=q^{r-1}(q-1)/2$ is odd, implying that $q\equiv 3\pmod 4$. Therefore (d) holds.

Suppose (d) holds. Then both cases imply that $\ord_{n}p=\lcm(\ord_4p, \ord_{q^r}p)=\varphi(q^r)=\varphi(n)/2$.

\smallskip
{\bf Case 4}. $n=q_1^{r_1}q_2^{r_2}$ where  $q_1,q_2$ are distinct odd primes and $r_1, r_2\in \N$.
\smallskip

Suppose $\varphi(n)=2\ord_np$.
Then \begin{align*}
\varphi(n)/2&=q_1^{r_1-1}(q_1-1)q_2^{r_2-1}(q_2-1)/2=\ord_np=\lcm(\ord_{q_1^{r_1}}p, \ord_{q_2^{r_2}}p)\\
&=\frac{\ord_{q_1^{r_1}}p \ord_{q_2^{r_2}}p}{\gcd(\ord_{q_1^{r_1}}p, \ord_{q_2^{r_2}}p)}\le \frac{q_1^{r_1-1}(q_1-1)q_2^{r_2-1}(q_2-1)}{\gcd(\ord_{q_1^{r_1}}p, \ord_{q_2^{r_2}}p)}\,.
\end{align*}

If $\gcd(\ord_{q_1^{r_1}}p, \ord_{q_2^{r_2}}p)=2$, then $\ord_{q_1^{r_1}}p=\varphi(q_1^{r_1})$, $\ord_{q_2^{r_2}}p=\varphi(q_2^{r_2})$, and  $\gcd(q_1^{r_1-1}(q_1-1), q_2^{r_2-1}(q_2-1))=2$, whence $q_1^{r_1-1}(q_1-1)/2$ is odd or $q_2^{r_2-1}(q_2-1)/2$ is odd. By symmetry, we can assume that $q_2\equiv 3\pmod 4$. This is the first case of (e).

If $\gcd(\ord_{q_1^{r_1}}p, \ord_{q_2^{r_2}}p)=1$, then one of the following holds
\begin{itemize}
	\item $\ord_{q_1^{r_1}}p=\varphi(q_1^{r_1})$,  $\ord_{q_2^{r_2}}p=\varphi(q_2^{r_2})/2$, and  $\gcd(q_1^{r_1-1}(q_1-1), q_2^{r_2-1}(q_2-1)/2)=1$, whence $q_2\equiv 3\pmod 4$.
	
	\item $\ord_{q_1^{r_1}}p=\varphi(q_1^{r_1})/2$,  $\ord_{q_2^{r_2}}p=\varphi(q_2^{r_2})$, and  $\gcd(q_1^{r_1-1}(q_1-1)/2, q_2^{r_2-1}(q_2-1))=1$, whence $q_1\equiv 3\pmod 4$.
\end{itemize}
By symmetry, we may assume that $\gcd(q_1^{r_1-1}(q_1-1), q_2^{r_2-1}(q_2-1)/2)=1$ and  $q_2\equiv 3\pmod 4$. This is the second case of (e).

Suppose (e) holds. Then $$\ord_np=\lcm(\ord_{q_1^{r_1}}p, \ord_{q_2^{r_2}}p)=\frac{\ord_{q_1^{r_1}}p \ord_{q_2^{r_2}}p}{\gcd(\ord_{q_1^{r_1}}p, \ord_{q_2^{r_2}}p)}=\varphi(n)/2\,.$$

\smallskip
{\bf Case 5}. $n=q_1^{r_1}q_2^{r_2}q_3^{r_3}m$, where $q_1,q_2, q_3$ are pairwise distinct odd primes and $r_1, r_2, r_3, m\in \N$ with $\gcd(q_1q_2q_3, m)=1$.
\smallskip

Suppose $\varphi(n)=2\ord_np$.
Then \begin{align*}
\varphi(n)/2&=q_1^{r_1-1}(q_1-1)q_2^{r_2-1}(q_2-1)q_3^{r_3-1}(q_3-1)\varphi(m)/2=\ord_np\\
&=\lcm(\ord_{q_1^{r_1}}p, \ord_{q_2^{r_2}}p, \ord_{q_3^{r_3}}p, \ord_mp)\\
&\le \lcm(\varphi(q_1^{r_1}), \varphi(q_2^{r_2}), \varphi(q_3^{r_3}), \varphi(m))\\
&=\lcm(q_1^{r_1-1}(q_1-1), q_2^{r_2-1}(q_2-1), q_3^{r_3-1}(q_3-1), \varphi(m))\\
&\le q_1^{r_1-1}(q_1-1)q_2^{r_2-1}(q_2-1)q_3^{r_3-1}(q_3-1)\varphi(m)/4\,,
\end{align*}
a contradiction.

%
%
%
%
%
%

\bigskip
\textbf{(3)} 	We distinguish seven cases depending on the prime divisors of $n$.

\smallskip
{\bf Case 1}. $n=2^r$ with $r\in \N$.
\smallskip

Suppose $\varphi(n)=4\ord_np\ge 4$. Then $r\ge 3$. If $n=8$, then $\ord_8p=1$ and hence  $p\equiv 1\pmod 8$.
If $n=16$, then $\ord_{16}p=2$ and hence $p\not\equiv 1\pmod {16}$ and $p^2\equiv 1\pmod {16}$, whence $p\equiv -1, \pm 7 \pmod {16}$. Then (a) holds.

If $r\ge 5$, then $\ord_{2^r}p=\varphi(2^r)/4=2^{r-3}$, whence $p^{4}\equiv 1\pmod {32}$ but $p^2\not\equiv 1\pmod {32}$. Therefore $p\equiv \pm7, \pm9\pmod {32}$ and hence (b) holds.

Suppose (a) holds. It is easy to see that $\ord_np=\varphi(n)/2$. Suppose (b) holds. Then
$p^4\equiv 1\pmod {32}$ and $p^2\not\equiv 1\pmod {32}$, whence $\ord_{2^5}p=2^{5-2}$ and hence $\ord_{2^r}p=2^{r-3}=\varphi(2^r)/4$.

\smallskip
{\bf Case 2}. $n=q^r$ with $q$ an odd prime and $r\in \N$.
\smallskip

Then $\ord_{n}p=\varphi(n)/4=q^{r-1}(q-1)/4$ if and only if  (c)  holds.

\smallskip
{\bf Case 3}. $n=q_1^{r_1}q_2^{r_2}$ with $q_1,q_2$ distinct odd primes and $r_1, r_2\in \N$.
\smallskip

By symmetry, we may assume that $\frac{\varphi(q_1^{r_1})}{\ord_{q_1^{r_1}}p}\le \frac{\varphi(q_2^{r_2})}{\ord_{q_2^{r_2}}p}$.

Suppose $\varphi(n)=4\ord_np$.
Then \begin{align*}
\varphi(n)/4&=q_1^{r_1-1}(q_1-1)q_2^{r_2-1}(q_2-1)/4=\ord_np=\lcm(\ord_{q_1^{r_1}}p, \ord_{q_2^{r_2}}p)\\
&=\frac{\ord_{q_1^{r_1}}p \ord_{q_2^{r_2}}p}{\gcd(\ord_{q_1^{r_1}}p, \ord_{q_2^{r_2}}p)}\le \frac{q_1^{r_1-1}(q_1-1)q_2^{r_2-1}(q_2-1)}{\gcd(\ord_{q_1^{r_1}}p, \ord_{q_2^{r_2}}p)}\,.
\end{align*}

If $\gcd(\ord_{q_1^{r_1}}p, \ord_{q_2^{r_2}p})=4$, then $\ord_{q_1^{r_1}}p=\varphi(q_1^{r_1})$ and $\ord_{q_2^{r_2}}p=\varphi(q_2^{r_2})$. This is the first case of (d).
If $\gcd(\ord_{q_1^{r_1}}p, \ord_{q_2^{r_2}p})=2$, then  $\ord_{q_1^{r_1}}p=\varphi(q_1^{r_1})$ and $\ord_{q_2^{r_2}}p=\varphi(q_2^{r_2})/2$. This is the second case of (d).
If $\gcd(\ord_{q_1^{r_1}}p, \ord_{q_2^{r_2}p})=1$, then one of the following holds
\begin{itemize}
	\item $\ord_{q_1^{r_1}}p=\varphi(q_1^{r_1})$,  $\ord_{q_2^{r_2}}p=\varphi(q_2^{r_2})/4$, and  $\gcd(q_1^{r_1-1}(q_1-1), q_2^{r_2-1}(q_2-1)/4)=1$.
	
	\item $\ord_{q_1^{r_1}}p=\varphi(q_1^{r_1})/2$,  $\ord_{q_2^{r_2}}p=\varphi(q_2^{r_2})/2$, and  $\gcd(q_1^{r_1-1}(q_1-1)/2, q_2^{r_2-1}(q_2-1)/2)=1$.
\end{itemize}
These are the third and fourth cases of (d).

Suppose (d) holds. Then $$\ord_np=\lcm(\ord_{q_1^{r_1}}p, \ord_{q_2^{r_2}}p)=\frac{\ord_{q_1^{r_1}}p \ord_{q_2^{r_2}}p}{\gcd(\ord_{q_1^{r_1}}p, \ord_{q_2^{r_2}}p)}=\varphi(n)/4\,.$$

\smallskip
{\bf Case 4}. $n=q_1^{r_1}q_2^{r_2}q_3^{r_3}$ with $q_1,q_2, q_3$  pairwise distinct odd primes and $r_1,r_2,r_3\in \N$.
\smallskip

By symmetry, we may assume that $\frac{\varphi(q_1^{r_1})}{\ord_{q_1^{r_1}}p}\le \frac{\varphi(q_2^{r_2})}{\ord_{q_2^{r_2}}p}\le \frac{\varphi(q_3^{r_3})}{\ord_{q_3^{r_3}}p}$.

Suppose $\varphi(n)=4\ord_np$.
Then \begin{align*}
\varphi(n)/4&=q_1^{r_1-1}(q_1-1)q_2^{r_2-1}(q_2-1)q_3^{r_3-1}(q_3-1)/4=\ord_np\\
&=\lcm(\ord_{q_1^{r_1}}p, \ord_{q_2^{r_2}}p, \ord_{q_3^{r_3}}p)\\
&\le \lcm(\varphi(q_1^{r_1}), \varphi(q_2^{r_2}), \varphi(q_3^{r_3}))\\
&=\lcm(q_1^{r_1-1}(q_1-1), q_2^{r_2-1}(q_2-1), q_3^{r_3-1}(q_3-1))\\
&\le q_1^{r_1-1}(q_1-1)q_2^{r_2-1}(q_2-1)q_3^{r_3-1}(q_3-1)/4\,,
\end{align*}
which implies  (e) holds.

Suppose (e) holds. Then it is easy to see that $\varphi(n)=4\ord_np$.

\smallskip
{\bf Case 5}. $n=q_1^{r_1}q_2^{r_2}q_3^{r_3}q_4^{r_4}m$, where  $q_1,q_2, q_3, q_4$ are pairwise distinct odd primes and $r_1,r_2,r_3,m\in \N$ with $\gcd(q_1q_2q_3q_4, m)=1$.
\smallskip

Suppose $\varphi(n)=4\ord_np$.
Then \begin{align*}
\varphi(n)/4&=q_1^{r_1-1}(q_1-1)q_2^{r_2-1}(q_2-1)q_3^{r_3-1}(q_3-1)q_4^{r_4-1}(q_4-1)\varphi(m)/4=\ord_np\\
&=\lcm(\ord_{q_1^{r_1}}p, \ord_{q_2^{r_2}}p, \ord_{q_3^{r_3}}p, \ord_{q_4^{r_4}}p,\ord_mp)\\
&\le \lcm(\varphi(q_1^{r_1}), \varphi(q_2^{r_2}), \varphi(q_3^{r_3}),\varphi(q_4^{r_4}), \varphi(m))\\
&=\lcm(q_1^{r_1-1}(q_1-1), q_2^{r_2-1}(q_2-1), q_3^{r_3-1}(q_3-1), q_4^{r_4-1}(q_4-1),\varphi(m))\\
&\le q_1^{r_1-1}(q_1-1)q_2^{r_2-1}(q_2-1)q_3^{r_3-1}(q_3-1)q_4^{r_4-1}(q_4-1)\varphi(m)/8\,,
\end{align*}
a contradiction.

\smallskip
{\bf Case 6}. $n=2^tq^r$ with $q$  an odd prime, $r\in\N$, and $t\ge 2$.
\smallskip

Suppose $\varphi(n)=4\ord_np$.
Then $\varphi(n)/4=2^{t-3}q^{r-1}(q-1)=\ord_{n}p=\lcm(\ord_{2^t}p, \ord_{q^r}p)$.

If $t=2$, then $q^{r-1}(q-1)/2=\lcm(\ord_4p, \ord_{q^r}p)$ and hence one of the following holds
\begin{itemize}
	\item $\ord_4p=1$ and $\ord_{q^r}p=q^{r-1}(q-1)/2$. This is the first case of (f).

	\item $\ord_4p=2$ and $\ord_{q^r}p=q^{r-1}(q-1)/2$ is even, whence $q\equiv 1\pmod 4$. This is the second case of (f).

	\item $\ord_4p=2$ and $\ord_{q^r}p=q^{r-1}(q-1)/4$ is odd, whence $q\equiv 5\pmod 8$. This is the third case of (f).
\end{itemize}

If $t=3$, then $q^{r-1}(q-1)=\lcm(\ord_8p, \ord_{q^r}p)$ and hence one of the following holds
\begin{itemize}
	\item  $\ord_{q^r}p=q^{r-1}(q-1)$. This is the first case of (h).

	\item $\ord_8p=2$ and $\ord_{q^r}p=q^{r-1}(q-1)/2$ is odd, whence $q\equiv 3\pmod 4$. This is the second case of (h).
\end{itemize}

If $t\ge 4$, then $$2^{t-3}q^{r-1}(q-1)=\lcm(\ord_{2^t}p, \ord_{q^r}p)=\frac{\ord_{2^t}p \ord_{q^r}p}{\gcd(\ord_{2^t}p, \ord_{q^r}p)}\le \frac{2^{t-2}q^{r-1}(q-1)}{\gcd(\ord_{2^t}p, \ord_{q^r}p)}$$ and hence one of the following holds.
\begin{itemize}
	\item $\gcd(\ord_{2^t}p, \ord_{q^r}p)=2$, $\ord_{2^t}p=2^{t-2}\ge 4$, and $\ord_{q^r}p=q^{r-1}(q-1)$, whence $\ord_{q^r}p/2=q^{r-1}(q-1)/2$ is odd, i.e., $q\equiv 3\pmod 4$.
	
	\item  $\gcd(\ord_{2^t}p, \ord_{q^r}p)=1$, either $\ord_{2^t}p=2^{t-3}$ and $\ord_{q^r}p=q^{r-1}(q-1)$, or  $\ord_{2^t}p=2^{t-2}$ and $\ord_{q^r}p=q^{r-1}(q-1)/2$. Since $\gcd(2^{t-3}, q^{r-1}(q-1))\ge 2$, the latter holds and hence $q^{r-1}(q-1)/2$ is old, i.e., $q\equiv 3\pmod 4$.
\end{itemize}
Note that $\ord_{2^t}p=2^{t-2}$ if and only if $\ord_{16}p=4$ if and only if $p\equiv \pm3,\pm5\pmod {16}$. Thus
both cases imply that (i) holds.

Suppose  (f), or (h), or (i) holds. It is easy to check that  $\varphi(n)=4\ord_np$.

\smallskip
{\bf Case 7}. $n=2^tq_1^{r_1}q_2^{r_2}m$ where $r_1,r_2\in \N$, $q_1$ and $q_2$ are distinct odd primes, $t\ge 2$, $m$ is odd, and $\gcd(q_1q_2,m)=1$.
\smallskip

Suppose $\varphi(n)=4\ord_np$.
Then \begin{align*}
&\varphi(n)/4=2^{t-3}q_1^{r_1-1}(q_1-1)q_2^{r_2-1}(q_2-1)\varphi(m)\\
=&\ord_np=\lcm(\ord_{2^t}p, \ord_{q_1^{r_1}}p, \ord_{q_2^{r_2}}p, \ord_mp)\\
\le &\lcm(\ord_{2^t}p, q_1^{r_1-1}(q_1-1), q_2^{r_2-1}(q_2-1), \varphi(m))\,.
\end{align*}
If $t\ge 3$, then $$\lcm(\ord_{2^t}p, q_1^{r_1-1}(q_1-1), q_2^{r_2-1}(q_2-1), \varphi(m))\le 2^{t-4}q_1^{r_1-1}(q_1-1)q_2^{r_2-1}(q_2-1)\varphi(m)<\varphi(n)/4\,,$$ a contradiction. Thus $t=2$.
If $m\ge 3$, then $\varphi(m)$ is even and
$$\lcm(\ord_{2^t}p, q_1^{r_1-1}(q_1-1), q_2^{r_2-1}(q_2-1), \varphi(m))\le  \frac{q_1^{r_1-1}(q_1-1)q_2^{r_2-1}(q_2-1)\varphi(m)}{4\ord_4p}\,,$$ whence $2\ord_{4}p\le 1$, a contradiction.

Therefore $n=4q_1^{r_1}q_2^{r_2}$ and hence
$q_1^{r_1-1}(q_1-1)q_2^{r_2-1}(q_2-1)/2=\lcm(\ord_{4}p, \ord_{q_1^{r_1}q_2^{r_2}}p)$. It follows that  one of the following holds.
\begin{itemize}
	\item $\ord_{q_1^{r_1}q_2^{r_2}}p=q_1^{r_1-1}(q_1-1)q_2^{r_2-1}(q_2-1)/2$. By 2(e), we obtain this is the first and second  cases of (g).
	
	\item  $\ord_4p=2$ and $\ord_{q_1^{r_1}q_2^{r_2}}p=q_1^{r_1-1}(q_1-1)q_2^{r_2-1}(q_2-1)/4$ is odd, whence $p\equiv 3\pmod 4$, $q_1\equiv 3\pmod 4$, and $q_2\equiv 3\pmod 4$. By 3(d), we obtain that this is the third case of (g).
\end{itemize}

Suppose (g) holds. It is easy to check that $\varphi(n)=4\ord_np$.
\end{proof}

\begin{proposition}\label{p3.2}
	Let $n$ be a positive integer and let $p$ be an odd prime with $p\nmid n$. \begin{enumerate}[label={(\arabic*)}, font={\bfseries}]
		\item $\ord_np=\varphi(n)/2$ is odd if and only if one of the following holds
		\begin{enumerate}
			\item $n=4$ and $p\equiv 1\pmod 4$;
			
			\item $n=q^r$ or $2q^r$, where $r\in \N$ and $q$ is an odd prime with $q\equiv 3\pmod 4$,  and  $\ord_{q^{\epsilon(r)}}p=q^{\epsilon(r)-1}(q-1)/2$.
		\end{enumerate}
		
%
%
%
%
%
%
%
%
%
%
%

		\item
	 $\ord_np=\varphi(n)/2$ and $\ord_mp=\varphi(m)$ for all divisor $m$ of $n$ with $n\neq m$ if and only if one of the following holds
	\begin{enumerate}
		\item $n=4$ and $p\equiv 1\pmod 4$; or $n=8$ and $p\equiv 3\pmod 4$.

		\item $n=q$ is an odd prime  and  $\ord_{q}p=(q-1)/2$.

		\item $n=4q$, where $q$ is an odd prime, $p\equiv 3\pmod 4$, and    $p$ is a primitive root of $q$.

			\item $n=q_1q_2$ is a product of  two distinct primes, $p$ is a primitive root of unity of both $q_1$ and $q_2$, and $\gcd(q_1-1,q_2-1)=2$.
	\end{enumerate}

\item $\ord_np=\varphi(n)/4$ is odd if and only if one of the following holds
\begin{enumerate}	
	
	\item $n=8$ and $p\equiv 1\pmod 8$.
	
	\item $n=q^r$ or $2q^r$, where $q$ is an odd prime with $q\equiv 5\pmod 8$,  and  $\ord_{q^{\epsilon(r)}}p=q^{\epsilon(r)-1}(q-1)/4$.
	
	\item $n=q_1^{r_1}q_2^{r_2}$ or $2q_1^{r_1}q_2^{r_2}$, where $q_1,q_2$ are distinct  odd primes with $q_1\equiv 3\pmod 4$ and $q_2\equiv 3\pmod 4$,  and one of the following holds
	\begin{itemize}
			\item $p$ is a primitive root of $q_1^{\epsilon(r_1)}$, $\ord_{q_2^{\epsilon(r_2)}}p=q_2^{\epsilon(r_2)-1}(q_2-1)/2$,  and $\gcd(q_1^{\epsilon(r_1)-1}(q_1-1), q_2^{\epsilon(r_2)-1}(q_2-1)/2)=2$.

		\item $\ord_{q_1^{\epsilon(r_1)}}p=q_1^{\epsilon(r_1)-1}(q_1-1)/2$, $\ord_{q_2^{\epsilon(r_2)}}p=q_2^{\epsilon(r_2)-1}(q_2-1)/2$, and $\gcd(q_1^{\epsilon(r_1)-1}(q_1-1)/2, q_2^{\epsilon(r_2)-1}(q_2-1)/2)=1$.
	\end{itemize}

%
%

	\item $n=4q^r$, where $q$ is a prime with $q\equiv 3\pmod 4$,
	 $p\equiv 1\pmod 4$, and $\ord_{q^r}p=q^{r-1}(q-1)/2$.
\end{enumerate}

	\item $\varphi(n)=4\ord_np$ and $\varphi(m)\neq 4\ord_mp$ for all divisor $m$ of $n$ with $m\neq n$ if and only if one of the following holds
\begin{enumerate}	
	
	\item $n=8$ and $p\equiv 1\pmod 8$; or $n=16$ and $p\equiv 7,9,15\pmod {16}$.
	
	\item $n=q$, where $q$ is an odd prime with $q\equiv 1\pmod 4$,  and  $\ord_{q}p=(q-1)/4$.
	
	\item $n=q_1q_2$, where $q_1,q_2$ are distinct  odd primes,  and  one of the following holds.
	\begin{itemize}
		\item $p$ is a primitive root of both $q_1, q_2$  and $\gcd(q_1-1, q_2-1)=4$.
		
		\item $p$ is a primitive root of $q_1$, $\ord_{q_2}p=(q_2-1)/2$,  and $\gcd(q_1-1, (q_2-1)/2)=2$.

		\item $\ord_{q_1}p=(q_1-1)/2$, $\ord_{q_2}p=(q_2-1)/2$, and $\gcd((q_1-1)/2, (q_2-1)/2)=1$.	
	\end{itemize}

	\item $n=q_1q_2q_3$, where $q_1,q_2, q_3$ are distinct  odd primes,  and  one of the following holds.
	\begin{itemize}
		\item $p$ is a primitive root of  $q_1, q_2, q_3$  and $\gcd(q_i-1, q_j-1)=2$  for all distinct $i,j\in \{1,2,3\}$.
		
		\item $p$ is a primitive root of $q_1, q_2$, $\ord_{q_3}p=(q_3-1)/2$ is odd,  and $\gcd(q_i-1, q_j-1)=2$ for all distinct $i,j\in \{1,2,3\}$.	
		
	\end{itemize}
	
	\item $n=4q$, where $q$ is a prime,  and one of the following holds
	
	\begin{itemize}
		\item $p\equiv 1\pmod 4$ and $\ord_{q}p=(q-1)/2$;
		
		\item $p\equiv 3\pmod 4$, $q\equiv 1\pmod 4$,  and $\ord_{q}p=(q-1)/2$;
	\end{itemize}
	
	\item $n=4q_1q_2$, where $q_1,q_2$ are distinct  odd primes with $q_2\equiv 3\pmod 4$,  and one of the following holds
	
	\begin{itemize}
		\item  $\gcd(q_1-1, q_2-1)=2$ and $p$ is a primitive root of both $q_1$ and $q_2$.

		\item $\gcd(q_1-1, q_2-1)=2$, $p\equiv 3\pmod 4$, $p$ is a primitive root of  $q_1$, $\ord_{q_2}p=(q_2-1)/2$,

		\item $\gcd((q_1-1)/2, (q_2-1)/2)=1$, $p\equiv 3\mod 4$, $q_1\equiv 3\pmod 4$,  $\ord_{q_1}p=(q_1-1)/2$ and  $\ord_{q_2}p=(q_2-1)/2$.
	\end{itemize}

	\item $n=8q$, where $q$ is an odd prime,  and  one of the following holds
	\begin{itemize}
		\item  $p$ is a primitive root of $q$;
		
		\item $p\not\equiv 1\pmod 8$, $q\equiv 3\pmod 4$, and $\ord_{q}p=(q-1)/2$.		
	\end{itemize}
	
\end{enumerate}
	
	\item $\varphi(n)=4\ord_np$ and $\varphi(m)=\ord_mp$ for all divisors $m$ of $n$ with $m\neq n$ if and only if one of the following holds.
	
	\begin{itemize}
	
		\item $n=q$, where $q$ is an odd prime with $q\equiv 1\pmod 4$,  and  $\ord_{q}p=(q-1)/4$.
		
		\item $n=q_1q_2$, where $q_1,q_2$ are distinct  odd primes,  and  $p$ is a primitive root of both $q_1, q_2$  and $\gcd(q_1-1, q_2-1)=4$.
	
	\end{itemize}

\end{enumerate}
\end{proposition}

\begin{proof}
If we know the factorization of $n$, it is easy to see when $\varphi(n)/2$ and $\varphi(n)/4$ are odd. Furthermore,
if we know $\ord_n p$, it is easy to obtain $\ord_mp$ for every divisor $m$ of $n$ with $m\neq n$.

\textbf{(1)} This is done by checking each item of Proposition \ref{p2.6}.2  to see when $\varphi(n)/2$ is odd.

\textbf{(2)}  	This is done by checking each item of Proposition \ref{p2.6}.2 to see when $p$ is a primitive root of  every divisor $m$ of $n$ with $m\neq n$.

\textbf{(3)}  This is done  by checking each item of Proposition \ref{p2.6}.3 to see when $\varphi(n)/4$ is odd.

\textbf{(4)}  	This is done by checking each item of Proposition \ref{p2.6}.3 to see when $\ord_mp\neq \varphi(m)/4$  for every divisor $m$   of $n$ with $m\neq n$.

\textbf{(5)} 	This is done by checking each item of Proposition \ref{p2.6}.3 to see when $p$ is a primitive root of  every divisor $m$ of $n$ with $m\neq n$.
	
%
%
%
%
%
%
%
%
%
\end{proof}

\bigskip
\section{Group rings over local subrings of algebraic number fields}
\medskip

	Let $K$ be an algebraic number field and let $\mathcal O$ its ring of integers. If $\mathfrak p\subset \mathcal O$ is a nonzero prime ideal with $2\Z=\mathfrak p\cap \Z$ and $G$ is a finite abelian group with odd $\exp(G)$, then it follows immediately from Proposition \ref{2.5}.3 that
	 the group ring $\mathcal O_{\mathfrak p}[G]$ is feebly clean if and only if
	  $\mathcal O_{\mathfrak p}[G]$ is  weakly clean if and only if $\mathcal O_{\mathfrak p}[G]$ is   clean.
When $K$ is a cyclotomic field or a quadratic field, a complete characterization of  when $\mathcal O_{\mathfrak p}[G]$ is  clean was given in \cite[Theorem 1.1 and Theorem 1.3]{Li-Zh20}. In particular, $\Z_{(2)}[G]$ is feebly clean if and only if $\Z_{(2)}[G]$ is weakly clean if and only if  $\Z_{(2)}[G]$ is  clean if and only if
$2$ is a primitive root of $\exp(G)$.

Next we let $\mathfrak p\subset \mathcal O$ be a nonzero prime ideal with $p\Z=\mathfrak p\cap \Z$, where $p$ is an odd prime.
The following theorem provides a complete characterization for  $\Z_{(p)}[G]$  to be clean, weakly clean and feebly clean, extending the main results of McGovern \cite{Mc18}.

\begin{theorem} \label{thm1}
	Let $p\neq 2$ be a prime and let $G$ be a finite abelian group with exponent $n$ and  $p\nmid n$.
	Then
	\begin{enumerate}[label={(\arabic*)}, font={\bfseries}]
		\item $\Z_{(p)}[G]$ is clean if and only if $\Z_{(p)}[\zeta_n]$ is clean if and only if $p$ is a primitive root of $n$.
		
		\item $\Z_{(p)}[G]$ is feebly clean but not clean if and only if $\Z_{(p)}[\zeta_n]$ is feebly clean but  not clean if and only if  $\varphi(n)=2\ord_np$ (for a more detailed characterization, see Proposition \ref{p2.6}.2).
		
		\item $\Z_{(p)}[G]$ is weakly clean but not clean if and only if $G\cong C_n$ and  one of the following holds
		\begin{enumerate}
			\item either $n=4$ and $p\equiv 1\pmod 4$, or $n=8$ and $p\equiv 3\pmod 4$.
			
			\item $n$ is an odd prime with $\ord_np=(n-1)/2$.

			\item $n=4q$ for some odd prime $q$ and $p$ is a primitive root of unity of $q$.
			
			\item $n=q_1q_2$ is a product of  two distinct primes, $p$ is a primitive root of unity of both $q_1$ and $q_2$, and $\gcd(q_1-1,q_2-1)=2$. 	
		\end{enumerate}
	\end{enumerate}
\end{theorem}

\begin{proof}
	\textbf{(1)} Since $\Phi_d(x)$ is a monic irreducible polynomial over $\Z_{(p)}$ and $|\overline{\Z_{(p)}}|=p$, it follows from Proposition \ref{2.5}.1 that $\Z_{(p)}[G]$ is clean if and only if $\deg(\Phi_d(x))=\ord_d(p)=\varphi(d)$ for all divisors $d$ of $n$  if and only if $p$ is a primitive root of $d$ for all divisors $d$ of $n$ if and only if $p$ is a primitive root of $n$ if and only if $\Z_{(p)}[\zeta_n]$ is clean.

	\textbf{(2)} Suppose $\Z_{(p)}[G]$ is feebly clean but not clean. Then by (1) we have $\Z_{(p)}[\zeta_n]$ is not clean and by Theorem \ref{t1}.2 and Lemma \ref{lemm1}.1 we have $\Z_{(p)}[\zeta_n]$ is feebly clean.
	
	Suppose $\Z_{(p)}[\zeta_n]$ is feebly clean but not clean.	
	It follows from
 Proposition \ref{2.5}.2  that
	 $\varphi(n)=2\ord_np$.

	Suppose $\varphi(n)=2\ord_np$. Then for all divisors $d$ of $n$, we have $\deg(\Phi_d(x))=\varphi(d)=\ord_dp=\ord_d|\overline{\Z_{(p)}}|$ or $\deg(\Phi_d(x))=\varphi(d)=2\ord_dp=2\ord_d|\overline{\Z_{(p)}}|$.
	It follows from
	Proposition \ref{2.5} that $\Z_{(p)}[\zeta_n]$ is not clean and $\Z_{(p)}[\zeta_d]$ is feebly clean for all divisors $d$ of $n$. Therefore Theorem \ref{t1} and Lemma \ref{lemm1}  imply that $\Z_{(p)}[G]$ is feebly clean but not clean.

	\textbf{(3)}  Suppose $\Z_{(p)}[G]$ is weakly clean but not clean. Then by (1), we have $\Z_{(p)}[\zeta_n]$ is not clean.  If $G$ is not a cyclic group, then $G$ has at least two different cyclic subgroups of order $n$, which implies that $\lambda(n)\ge 2$. Thus by Theorem \ref{t1}.2 and Lemma \ref{lemm2}.2, we have  $\Z_{(p)}[G]$ is not weakly clean,  a contradiction. Therefore $G\cong C_n$. Since $\Phi_d(x)$ is a monic irreducible polynomial over $\Z_{(p)}$, we have $\lambda(d)=1$ for all divisors $d$ of $n$. Again by Theorem \ref{t1}.2 and Lemma \ref{lemm2}.2, we have  $\Z_{(p)}[\zeta_d]$ is clean for all divisors $d$ of $n$ with $d\neq n$. Then Proposition \ref{2.5} implies that
	$\varphi(n)=2\ord_np$ and $\varphi(d)=\ord_dp$ for every divisor $d$ of $n$ with $d\neq n$.
	The assertions follow from Proposition \ref{p3.2}.

%
	
	Suppose $G\cong C_n$ and one of (a), (b), (c), and (d) holds. It is easy to see that $\varphi(n)=2\ord_np$ and $\varphi(d)=\ord_dp$ for divisor $d$ of $n$ with $d\neq n$. It follows from Lemma \ref{lemm2}.2, Theorem \ref{t1}.2, and Proposition \ref{2.5} that  $\Z_{(p)}[G]$ is weakly clean but not clean.
\end{proof}


\medskip

By using Theorems 3.1.2 and 3.1.3.b, and \cite[Theorem 3.4]{Li-Zh20},  we obtain the following corollary which gives a summary of the main results of McGovern \cite{Mc18}.

\begin{corollary} \label{new3.2}

Let $p, q \in \N$ be distinct prime with $p> 2$. Then $\Z_{(p)}[C_q]$ is clean if and only if  $\ord_q(p)=q-1$. Moreover, the following statements are equivalent:

\begin{enumerate}[label={(\arabic*)}, font={\bfseries}]
\item $\Z_{(p)}[C_q]$ is weakly clean, but not clean.

\item $\Z_{(p)}[C_q]$ is feebly clean, but not clean.

\item  $\ord_q(p)=\frac{q-1}{2}$.

\end{enumerate}
\end{corollary}

We note that the above mentioned result of McGovern \cite{Mc18}, which asserts that $\Z_{(p)}[C_q]$ is weakly clean iff it is feebly clean, is no longer the case when $C_q$ is replaced by a general finite abelian group $G$. The following example shows that these three families of group rings $\Z_{(p)}[G]$  are properly contained within each other, that is,

Class of clean group rings  $\Z_{(p)}[G]$ $\not\subseteq $ Class of weakly clean group rings  $\Z_{(p)}[G]$  $\not\subseteq $  Class of feebly clean group rings.

\begin{example}
 Let $p, q\in \N$ be distinct prime with $p>2$ and $\ord_{q^2}p=q(q-1)/2$. Let $G$ be a finite abelian group of $\exp(G)=q^r$ for some $r\in \N$. It follows from Theorem \ref{thm1} that
		\begin{itemize}
			\item $\Z_{(p)}[G]$ is feebly  clean but not clean.
			
			\item $\Z_{(p)}[G]$ is weakly  clean if and only if $G$ is cyclic and $r=1$.	
		\end{itemize}
\end{example}

\medskip
Let $K$ be an algebraic number field, $\mathcal O$ its ring of integers, and $\mathfrak p\subset \mathcal O$ a nonzero prime ideal. Then there exists a prime $p$ such that $\mathfrak p\cap \Z=p\Z$ and the localization $\mathcal O_{\mathfrak p}$ is a discrete valuation ring, which implies that $\mathcal O_{\mathfrak p}[x]$ is a unique factorization domain (UFD).
Furthermore, the norm $N(\mathfrak p)=|\mathcal O/\mathfrak p|=|\overline{\mathcal O_{\mathfrak p}}|$ is a prime power of $p$. The following  lemma  will be used frequently  in sequel without further mention.\\

\begin{lemma}\label{l1}
Let $K=\Q(\zeta_m)$ be a cyclotomic field for some $m\in \N$, $\mathcal O=\Z[\zeta_m]$ its ring of integers, and $\mathfrak p\subset \mathcal O$ a nonzero prime ideal with $\mathfrak p\cap \Z=p\Z$ for some prime $p$. Suppose  $p\nmid m$. Then $ N(\mathfrak p)=p^{\ord_{m}p}$.
\end{lemma}

\begin{proof}
This follows from \cite[VI.1.12 and VI.1.15]{Fr-Ta92}.
\end{proof}

We now establish an explicit characterization for the group ring $\mathcal O_{\mathfrak p}[G]$ to be clean, weakly clean or feebly clean, when $K=\Q(\zeta_m)$ is a cyclotomic field extension of $\Q$.\\

\begin{theorem}\label{main1}
Let $K=\Q(\zeta_m)$ be a cyclotomic field for some $m\in \N$, $\mathcal O=\Z[\zeta_m]$ its ring of integers, $\mathfrak p\subset \mathcal O$ a nonzero prime ideal,
and $p\neq 2$  the prime with $p\Z=\mathfrak p\cap \Z$. Let  $G$ be a finite abelian group with $n=\exp(G)$ and $p\nmid n$. Suppose $n_1$ is the maximal divisor of $n$ such that $\gcd(m, n_1)=1$ and $n'=\frac{\lcm(m,n)}{mn_1}$.

\begin{enumerate}[label={(\arabic*)}, font={\bfseries}]
	\item  The group ring $\mathcal O_{\mathfrak p}[G]$ is  clean if and only if $\mathcal O_{\mathfrak p}[\zeta_n]$ is clean
	if and only if $$n'\varphi(n_1)\ord_mp=\lcm(\ord_{n'm}p, \ord_{n_1}p)\,.$$

	\item The group ring $\mathcal O_{\mathfrak p}[G]$ is feebly clean but not clean
	if and only if $\mathcal O_{\mathfrak p}[\zeta_n]$ is feebly clean but not clean if and only if $n'\varphi(n_1)\ord_mp=2\lcm(\ord_{n'm}p, \ord_{n_1}p)$ if and only if  one of the following holds
	\begin{enumerate}
		\item $\ord_{n_1}p=\varphi(n_1)$, $\ord_{n'm}p=n'\ord_{m}p$, and $\gcd(\ord_{n_1}p,\ord_{n'm}p)=2$.

		\item $\ord_{n_1}p=\varphi(n_1)/2$, $\ord_{n'm}p=n'\ord_{m}p$, and $\gcd(\ord_{n_1}p,\ord_{n'm}p)=1$.
		
		\item $\ord_{n_1}p=\varphi(n_1)$, $\ord_{n'm}p=(n'\ord_{m}p)/2$, and $\gcd(\ord_{n_1}p,\ord_{n'm}p)=1$ (note that in this case $n'$ must be even).
	\end{enumerate}

\item  The group ring $\mathcal O_{\mathfrak p}[G]$ is weakly clean but not clean
if and only if $ G\cong C_n$, $\Phi_n(x)$ is irreducible over $\mathcal O_{\mathsf p}$, and one of the following holds
\begin{enumerate}
	\item $\ord_{n_1}p=\varphi(n_1)$, $\ord_{n'm}p=n'\ord_{m}p$, $\gcd(\ord_{n_1}p,\ord_{n'm}p)=2$, and $\gcd(\ord_{d_1}p,\ord_{d'm}p)=1$ for every divisor $d$ of $n$.

	\item $\ord_{n_1}p=\varphi(n_1)/2$, $\ord_{n'm}p=n'\ord_{m}p$,  $\gcd(\ord_{n_1}p,\ord_{n'm}p)=1$, and $\ord_{d_1}p=\varphi(d_1)$ for divisor $d$ of $n$.
	
	\item $\ord_{n_1}p=\varphi(n_1)$, $n'=2$, $\ord_{2m}p=\ord_{m}p$, and $\gcd(\ord_{n_1}p,\ord_{m}p)=1$.
\end{enumerate}

\end{enumerate}
\end{theorem}

\medskip
\begin{proof}
Since $p\neq 2$ and $p\nmid n$, we have $\{2,d\}\subset  U(\mathcal O_{\mathfrak p})$ for all divisors $d$ of $\exp(G)$. Let $d$ be a divisor of $n$, let $d_1$ be the maximal divisor of $d$ such that $\gcd(d_1,m)=1$, and let $d'=\frac{\lcm(m,d)}{md_1}$.

Note that $$[\Q(\zeta_{m})(\zeta_{d}):\Q(\zeta_{m})]=[\Q(\zeta_{\lcm(m,d)}:\Q(\zeta_{m})]=\frac{\varphi(\lcm(m,d))}{\varphi(m)}=d'\varphi(d_1)$$ and
\begin{align*}
&\ord_{d} N(\mathfrak p)=\ord_{d}p^{\ord_{m}p}=\frac{\ord_{d}p}{\gcd(\ord_{d}p, \ord_{m}p)}\\
=&\frac{\lcm(\ord_{m}p, \ord_{d}p)}{\ord_{m}p}=\frac{\ord_{\lcm(d,m)}p}{\ord_{m}p}
=\frac{\lcm(\ord_{d'm}p, \ord_{d_1}p)}{\ord_{m}p}\,.
\end{align*}

\smallskip	
\textbf{(1)} Suppose $\mathcal O_{\mathsf p}[G]$ is clean. Then Lemma \ref{lemm1} and Theorem \ref{t1} imply that $\mathcal O_{\mathsf p}[\zeta_n]$ is clean.

Suppose $\mathcal O_{\mathsf p}[\zeta_n]$ is clean. Then Proposition \ref{2.5} implies that $[\Q(\zeta_m)(\zeta_n):\Q(\zeta_m)]=\deg(\phi_n(x))=\ord_n N(\mathsf p)$. Therefore $n'\varphi(n_1)\ord_mp=\lcm(\ord_{n'm}p, \ord_{n_1}p)$.

Suppose $n'\varphi(n_1)\ord_mp=\lcm(\ord_{n'm}p, \ord_{n_1}p)$. Then
 $$n'\varphi(n_1)\ord_mp=\frac{\ord_{n'm}p\  \ord_{n_1}p}{\gcd(\ord_{n'm}p, \ord_{n_1}p)}\le \frac{n'\ord_{m}p\ \varphi(n_1)}{\gcd(\ord_{n'm}p, \ord_{n_1}p)}\,,$$
whence \[\gcd(\ord_{n'm}p, \ord_{n_1}p)=1\ \text{ and }\ \ord_{n'm}p\,\ord_{n_1}p=n'\ord_{m}p\ \varphi(n_1)\,.\]
 Therefore for every divisor $d$ of $n$, we have \[
 \gcd(\ord_{d'm}p, \ord_{d_1}p)=1\ \text{ and }\ \ord_{d'm}p\,\ord_{d_1}p=d'\ord_{m}p\ \varphi(d_1)\,.
 \] It follows that $d'\varphi(d_1)\ord_mp=\lcm(\ord_{d'm}p, \ord_{d_1}p)$ for all divisors $d$ of $n$, whence $$[\Q(\zeta_m)(\zeta_d):\Q(\zeta_m)]=\deg(\phi_d(x))=\ord_d N(\mathsf p)$$ for all divisors $d$ of $n$. By Proposition \ref{2.5}, we have $\mathcal O_{\mathsf p}[G]$ is clean.

\smallskip		
\textbf{(2)} Note that $2\in U(\mathcal O_{\mathsf p})$.
By Proposition \ref{2.5} we have  $\mathcal O_{\mathfrak p}[G]$ is feebly clean but not clean if and only if
\[\tag{*}
\begin{aligned}
&d'\varphi(d_1)\ord_mp\le 2\lcm(\ord_{d'm}p, \ord_{d_1}p) \text{ for all divisors $d$ of $n$}\\
\text{ and }\quad & n'\varphi(n_1)\ord_mp=2\lcm(\ord_{n'm}p, \ord_{n_1}p)\,.
\end{aligned}
\]

Suppose that $\mathcal O_{\mathfrak p}[G]$ is feebly clean but not clean. Then by (1), $\mathcal O_{\mathsf p}[\zeta_n]$ is not clean and by Lemma \ref{lemm1} and Theorem \ref{t1}, $\mathcal O_{\mathsf p}[\zeta_n]$ is feebly clean.

Suppose that $\mathcal O_{\mathsf p}[\zeta_n]$ is feebly clean but not clean. By Proposition \ref{2.5}.2, we have $[\Q(\zeta_m)(\zeta_n):\Q(\zeta_m)]=\deg(\phi_n(x))=2\ord_n N(\mathsf p)$ and whence $n'\varphi(n_1)\ord_mp=2\lcm(\ord_{n'm}p, \ord_{n_1}p)$.

Suppose $n'\varphi(n_1)\ord_mp=2\lcm(\ord_{n'm}p, \ord_{n_1}p)$.  Since \begin{align*}
n'\varphi(n_1)= 2\frac{\lcm(\ord_{n'm}p, \ord_{n_1}p)}{\ord_{m}p}&=\frac{2}{\gcd(\ord_{n'm}p, \ord_{n_1}p)}\cdot\frac{\ord_{n'm}p\,\ord_{n_1}p}{\ord_{m}p}\\
&\le \frac{2}{\gcd(\ord_{n'm}p, \ord_{n_1}p)}\cdot\frac{n'\ord_{m}p\,\ord_{n_1}p}{\ord_{m}p}\\
&\le \frac{2}{\gcd(\ord_{n'm}p, \ord_{n_1}p)}\cdot n'\varphi(n_1)\,,
\end{align*}
we obtain that one of the following holds
\begin{enumerate}
 \item[(a)] $\ord_{n_1}p=\varphi(n_1)$, $\ord_{n'm}p=n'\ord_{m}p$, and $\gcd(\ord_{n_1}p,\ord_{n'm}p)=2$.

 \item[(b)] $\ord_{n_1}p=\varphi(n_1)/2$, $\ord_{n'm}p=n'\ord_{m}p$, and $\gcd(\ord_{n_1}p,\ord_{n'm}p)=1$.

\item[(c)] $\ord_{n_1}p=\varphi(n_1)$, $\ord_{n'm}p=(n'\ord_{m}p)/2$, and $\gcd(\ord_{n_1}p,\ord_{n'm}p)=1$ (this can  not happen if $2\nmid n'$).
 \end{enumerate}

Suppose (a), or (b), or (c) holds.
If (a) holds, then  $\ord_{d_1}p=\varphi(d_1)$, $\ord_{d'm}p=d'\ord_mp$, and  $\gcd(\ord_{d_1}p,\ord_{d'm}p)\le 2$, whence $$2\lcm(\ord_{d'm}p, \ord_{d_1}p)=\frac{2\ord_{d'm}p\,\ord_{d_1}p}{\gcd(\ord_{d'm}p, \ord_{d_1}p)}\ge d'\ord_mp \ \varphi(d_1)\,.$$
If (b) holds, then $\ord_{d_1}p\ge \varphi(d_1)/2$, $\ord_{d'm}p=d'\ord_mp$, and  $\gcd(\ord_{d_1}p,\ord_{d'm}p)=1$, whence $$2\lcm(\ord_{d'm}p, \ord_{d_1}p)=\frac{2\ord_{d'm}p\,\ord_{d_1}p}{\gcd(\ord_{d'm}p, \ord_{d_1}p)}\ge  d'\ord_mp \ \varphi(d_1)\,.$$
If (c) holds, then $\ord_{d_1}p=\varphi(d_1)$, $\ord_{d'm}p= (d'\ord_mp)/2$, and  $\gcd(\ord_{d_1}p,\ord_{d'm}p)=1$, whence $$2\lcm(\ord_{d'm}p, \ord_{d_1}p)=\frac{2\ord_{d'm}p\,\ord_{d_1}p}{\gcd(\ord_{d'm}p, \ord_{d_1}p)}\ge  d'\ord_mp \ \varphi(d_1)\,.$$
All those three cases imply that $(*)$ holds and hence $\mathcal O_{\mathsf p}[G]$ is clean.

\medskip	
\textbf{(3)} Suppose $\mathcal O_{\mathsf p}[G]$ is weakly clean but not clean. Then by (1) we have $\mathcal O_{\mathsf p}[\zeta_n]$ is not clean. Assume to the contrary that $G$ is not cyclic or $\Phi_n(x)$ is not irreducible over $\mathcal O_{\mathsf p}$. Then $\lambda(n)\ge 2$. In view of Theorem \ref{t1} and Lemma \ref{lemm1}, we have  $\mathcal O_{\mathsf p}[G]$ is not weakly clean, a contradiction. Thus $G\cong C_n$ and $\Phi_n(x)$ is  irreducible over $\mathcal O_{\mathsf p}$. It  follows from Theorem \ref{t1} and Lemma \ref{lemm1} that $\mathcal O_{\mathsf p}[\zeta_d]$ is clean for all divisors $d$ of $n$ with $d\neq n$. The assertion follows from (1) and (2).

Suppose $G\cong C_n$, $\Phi_n(x)$ is  irreducible over $\mathcal O_{\mathsf p}$ and one of (a), (b), (c) holds.
Then $\lambda(n)=1$ and (2) implies that $\mathcal O_{\mathsf p}[G]$ is feebly clean but not clean. Since any of (a), (b), (c) implies that $d'\varphi(d_1)\ord_mp=\lcm(\ord_{d'm}p,\ord_{d_1}p)$ for all divisors $d$ of $n$ with $d\neq n$, we have $\mathcal O_{\mathsf p}[\zeta_d]$ is clean for all divisors $d$ of $n$ with $d\neq n$. It follows from Theorem \ref{t1} and Lemma \ref{lemm1} that $\mathcal O_{\mathsf p}[G]$ is weakly clean.
\end{proof}

\medskip
Next we investigate when a group ring over a local subring of a quadratic field is clean, weakly clean and feebly clean, and our main result is stated in Theorem \ref{main2} below. Let $d$ be a non-zero square-free integer with $d\neq 1$,  $K=\Q(\sqrt{d})$ a quadratic number field,
\begin{align*}
\omega=\left\{\begin{aligned}
&\sqrt{d}      &&\text{ if }d\equiv 2,3\pmod 4\,,\\
&\frac{1+\sqrt{d}}{2} &&\text{ if }d\equiv 1\pmod 4\,,
\end{aligned}
\right.
\text{ and }
\Delta=\left\{\begin{aligned}
&4d      &&\text{ if }d\equiv 2,3\pmod 4\,,\\
&d &&\text{ if }d\equiv 1\pmod 4\,.
\end{aligned}
\right.
\end{align*}
Then $\mathcal O_K=\Z[\omega]$ is the ring of integers  and $\Delta$ is the discriminant of $K$.

 For an odd prime $p$ and an integer $a$, we denote by $\left(\frac{a}{p}\right)\in \{-1,0,1\}$ the Legendre symbol of $a$ modulo $p$.

We first provide two useful lemmas.

\begin{lemma}\label{3.4}
Let $d\neq 1$ be a non-zero square-free integer and let $\Delta$ be the discriminant of $\Q(\sqrt{d})$.
Then $\Q(\sqrt{d})\subset \Q(\zeta_n)$ if and only if $n$ is a multiple of $\Delta$.
\end{lemma}

\begin{proof}
This follows from  \cite[Corollary 4.5.5]{We06}
\end{proof}

\begin{lemma}\label{3.5}
Let $d\neq 1$ be a non-zero square-free integer and let $I$ be a prime ideal of $\mathcal O_K$, where $K=\Q(\sqrt{d})$.
Suppose $\Delta$ is the discriminant of $K$ and $\ch \mathcal O_K/I=p$, where $p$ is an odd prime.
Then $N(I)=p$ if and only if $\big(\frac{\Delta}{p}\big)=1$ or $0$.
\end{lemma}

\begin{proof}
This follows from \cite[Theorem 22, III.2.1, and V.1.1]{Fr-Ta92}.
\end{proof}

\begin{theorem}\label{main2}
Let $K=\Q(\sqrt{d})$ be a quadratic field for some non-zero square-free integer $d\neq 1$, $\mathcal O$ its ring of integers, $\mathfrak p\subset \mathcal O$ a nonzero prime ideal with $ \mathfrak p\cap \Z=p\Z$ and $p\neq 2$, and $G$ a  finite nontrivial abelian group with $p\nmid \exp(G)$.  Let $\Delta$ be the  discriminant of $K$.
 \begin{enumerate}[label={(\arabic*)}, font={\bfseries}]
 \item  Suppose $\Delta\nmid \exp(G)$.
 \begin{enumerate}
 	\item $\mathcal O_{\mathfrak p}[G]$ is  clean  if and only if $\mathcal O_{\mathsf p}[\zeta_{\exp(G)}]$ is clean if and only if one of the following holds
 	\begin{enumerate}
 		\item $\varphi(\exp(G))=\ord_{\exp(G)}p$ and $\big(\frac{\Delta}{p}\big)=1$ or $0$.

 		\item $\exp(G)=2$ and $\big(\frac{\Delta}{p}\big)=-1$.
 	\end{enumerate} 
 	
 	\item $\mathcal O_{\mathfrak p}[G]$ is feebly clean but not clean  if and only if $\mathcal O_{\mathsf p}[\zeta_{\exp(G)}]$ is feebly clean but not clean if and only if one of the following holds  \begin{enumerate}
 		\item $\varphi(\exp(G))=2\ord_{\exp(G)}p$ and $\big(\frac{\Delta}{p}\big)=1$ or $0$;

 		\item $\varphi(\exp(G))/2=\ord_{\exp(G)}p$ is odd and $\big(\frac{\Delta}{p}\big)=-1$;

 		\item $\exp(G)\neq 2$,  $\varphi(\exp(G))=\ord_{\exp(G)}p$, and $\big(\frac{\Delta}{p}\big)=-1$:
 	\end{enumerate} 

 \item $\mathcal O_{\mathfrak p}[G]$ is weekly clean but not clean  if and only if $G\cong C_{\exp(G)}$, $\Phi_{\exp(G)}(x)$ is irreducible over $\mathcal O_{\mathfrak p}$, and either  $\exp(G)=4$ and $\big(\frac{\Delta}{p}\big)=-1$, or  $\big(\frac{\Delta}{p}\big)\neq -1$ and one of the following holds  \begin{enumerate}
 	\item $\exp(G)=4$ and $p\equiv 1\pmod 4$; or $\exp(G)=8$ and $p\equiv 3\pmod 4$.

 	\item $\exp(G)=q$ is an odd prime  and  $\ord_{q}p=(q-1)/2$.

 	\item $\exp(G)=4q$, where $q$ is an odd prime,  and    $p$ is a primitive root of $q$.

 	\item $\exp(G)=q_1q_2$ is a product of  two distinct primes, $p$ is a primitive root of unity of both $q_1$ and $q_2$, and $\gcd(q_1-1,q_2-1)=2$.
 \end{enumerate} 
 \end{enumerate}

\item  Suppose  $\Delta\mid \exp(G)$ and $\big(\frac{\Delta}{p}\big)=1$ or $0$. Then

\begin{enumerate}
	\item $\mathcal O_{\mathfrak p}[G]$ is  clean  if and only if $\mathcal O_{\mathsf p}[\zeta_m]$ is clean for all divisors $m$ of $n$ if and only if one of the following holds
	\begin{enumerate}
		\item $\exp(G)=|\Delta|=8$, $d=\pm 2$, and $p\equiv 3\pmod 4$.

		\item $\exp(G)=2^r$ with $r\ge 4$, $|\Delta|=8$, $d=\pm 2$, and $p\equiv 3,11\pmod {16}$.
		
		\item $\exp(G)=q^r$ or $2q^r$, $|\Delta|=q$, $d=\pm q\equiv 1\pmod 4$, and $\ord_{q^{\epsilon(r)}}p=q^{\epsilon(r)-1}(q-1)/2$.

			\item $\exp(G)=4q^r$, $|\Delta|=q$, $d=- q\equiv 1\pmod 4$, $p\equiv 3\pmod 4$ and $\ord_{q^{\epsilon(r)}}p=q^{\epsilon(r)-1}(q-1)/2$.

				\item $\exp(G)=4q^r$, $|\Delta|=4q$, $d=\pm q\equiv 3\pmod 4$, $p\equiv 3\pmod 4$, and $p$ is a primitive root of $q^{\epsilon(r)}$.

					\item $\exp(G)=q_1^{r_1}q_2^{r_2}$ or $2q_1^{r_1}q_2^{r_2}$, $|\Delta|=q_2$, $d=- q_2\equiv 1\pmod 4$, $p$ is a primitive root of $q_1^{\epsilon(r_1)}$, $\ord_{q_2^{\epsilon(r_2)}}p=q_2^{\epsilon(r_2)-1}(q_2-1)/2$, and $\gcd(q_1^{\epsilon(r_1)-1}(q_1-1), q_2^{\epsilon(r_2)-1}(q_2-1))=2$.
					
					\item $\exp(G)=q_1^{r_1}q_2^{r_2}$ or $2q_1^{r_1}q_2^{r_2}$, $|\Delta|=q_1q_2$, $d=\pm q_1q_2\equiv 1\pmod 4$, $p$ is a primitive root of $q_1^{\epsilon(r_1)}$ and $q_2^{\epsilon(r_2)}$,  and $\gcd(q_1^{\epsilon(r_1)-1}(q_1-1), q_2^{\epsilon(r_2)-1}(q_2-1))=2$

	\end{enumerate}
	\item $\mathcal O_{\mathfrak p}[G]$ is feebly clean  if and only if $\mathcal O_{\mathsf p}[\zeta_{m}]$ is feebly clean for all divisors $m$ of $\exp(G)$ if and only if
	$\ord_{\exp(G)}p=\varphi(\exp(G))/2$ or one of the following holds
	\begin{enumerate}
		\item $\exp(G)=|\Delta|=8$, $d=\pm 2$, and $p\equiv 1\pmod 8$;
		
		\item $\exp(G)=2^r$ with $r\ge 4$, $|\Delta|=8$, $d=\pm 2$, and $p\equiv 9\pmod {16}$.

		\item $\exp(G)=q^r$ or $2q^r$ with $q\equiv 1\pmod 4$, $\Delta=d=q$, and $\ord_{q^{\epsilon(r)}}p=q^{\epsilon(r)-1}(q-1)/4$.

		\item $\exp(G)=q_1^{r_1}q_2^{r_2}$ or $2q_1^{r_1}q_2^{r_2}$ and one of the following  holds.

		\begin{itemize}
			\item $p$ is a primitive root of both $q_1^{\epsilon(r_1)}, q_2^{\epsilon(r_2)}$, $\gcd(q_1^{\epsilon(r_1)-1}(q_1-1), q_2^{\epsilon(r_2)-1}(q_2-1))=4$, and $\Delta=d\in \{q_1, q_2, q_1q_2\}$.
			
			\item $p$ is a primitive root of $q_1^{\epsilon(r_1)}$, $\ord_{q_2^{\epsilon(r_2)}}p=q_2^{\epsilon(r_2)-1}(q_2-1)/2$,  $\gcd(q_1^{\epsilon(r_1)-1}(q_1-1), q_2^{\epsilon(r_2)-1}(q_2-1)/2)=2$, $|\Delta|\in \{q_1, q_2, q_1q_2\}$, and $d=\Delta\equiv 1\pmod 4$.

			\item $\ord_{q_1^{\epsilon(r_1)}}p=q_1^{\epsilon(r_1)-1}(q_1-1)/2$, $\ord_{q_2^{\epsilon(r_2)}}p=q_2^{\epsilon(r_2)-1}(q_2-1)/2$, $\gcd(q_1^{\epsilon(r_1)-1}(q_1-1)/2, q_2^{\epsilon(r_2)-1}(q_2-1)/2)=1$,  $|\Delta|\in \{q_1, q_2, q_1q_2\}$, and  $d=\Delta\equiv 1\pmod 4$.

			\item $p$ is a primitive root of $q_1^{\epsilon(r_1)}$, $\ord_{q_2^{\epsilon(r_2)}}p=q_2^{\epsilon(r_2)-1}(q_2-1)/4$,  $\gcd(q_1^{\epsilon(r_1)-1}(q_1-1), q_2^{\epsilon(r_2)-1}(q_2-1)/4)=1$, and  $d=\Delta=q_2$.
		\end{itemize}
		
%

		\item $\exp(G)=q_1^{r_1}q_2^{r_2}q_3^{r_3}$ or $2q_1^{r_1}q_2^{r_2}q_3^{r_3}$ and one of the following holds
		\begin{itemize}
			\item $p$ is a primitive root of  $q_1^{r_1}, q_2^{r_2}, q_3^{r_3}$,  $\gcd(q_i^{r_i-1}(q_i-1), q_j^{r_j-1}(q_j-1))=2$  for all distinct $i,j\in \{1,2,3\}$, $|\Delta|\in \{q_1, q_2, q_3, q_1q_2, q_1q_3, q_2q_3, q_1q_2q_3\}$, and $d=\Delta\equiv 1\pmod 4$.
			
			\item $p$ is a primitive root of $q_1^{r_1}, q_2^{r_2}$, $\ord_{q_3^{r_3}}p=q_3^{r_3-1}(q_3-1)/2$ is odd,   $\gcd(q_i^{r_i-1}(q_i-1), q_j^{r_j-1}(q_j-1))=2$ for all distinct $i,j\in \{1,2,3\}$, $|\Delta|\in \{q_1, q_2, q_3, q_1q_2, q_1q_3, q_2q_3, q_1q_2q_3\}$, and $d=\Delta\equiv 1\pmod 4$.

			\item $p$ is a primitive root of $q_1^{r_1}$, $\ord_{q_2^{r_2}}p=q_2^{r_2-1}(q_2-1)/2$ is odd, $\ord_{q_3^{r_3}}p=q_3^{r_3-1}(q_3-1)/2$ is odd,  $\gcd(q_i^{r_i-1}(q_i-1), q_j^{r_j-1}(q_j-1))=2$ for all distinct $i,j\in \{1,2,3\}$, and  $d=\Delta\in \{-q_2, -q_3, q_2q_3\}$.
		\end{itemize}

		\item $\exp(G)=4q^r$ and one of the following holds
		
		 \begin{itemize}
			\item $p\equiv 1\pmod 4$, $\ord_{q^r}p=q^{r-1}(q-1)/2$, and either $d=\Delta=\pm q\equiv 1\pmod 4$, or $d=\Delta/4=\pm q\equiv 3\pmod 4$.
			
			\item $p\equiv 3\pmod 4$, $q\equiv 1\pmod 4$,   $\ord_{q^r}p=q^{r-1}(q-1)/2$, and  either $d=\Delta=q$, or $d=\Delta/4=-q$.
			
			\item $p\equiv 3\pmod 4$, $q\equiv 5\pmod 8$,  $\ord_{q^r}p=q^{r-1}(q-1)/4$, and hence $d=\Delta=q$.
		\end{itemize}
%
%

		 \item $\exp(G)=4q_1^{r_1}q_2^{r_2}$ with $q_2\equiv 3\pmod 4$ and one of the following holds
		
		 \begin{itemize}
		 	
		 	\item  $\gcd(q_1^{\epsilon(r_1)-1}(q_1-1), q_2^{\epsilon(r_2)-1}(q_2-1))=2$,  $p$ is a primitive root of both $q_1^{\epsilon(r_1)},q_2^{\epsilon(r_2)}$,  $|\Delta|\in \{q_1, q_2, 4q_1, 4q_2, q_1q_2, 4q_1q_2\}$, and either $d=\Delta\equiv 1\pmod 4$, or $d=\Delta/4\equiv 3\pmod 4$.

		 	\item   $\gcd(q_1^{\epsilon(r_1)-1}(q_1-1), q_2^{\epsilon(r_2)-1}(q_2-1))=2$,  $p$ is a primitive root of $q_1^{\epsilon(r_1)}$,  $\ord_{q_2^{\epsilon(r_2)}}p=q_2^{\epsilon(r_2)-1}(q_2-1)/2$, and
		 	either $p\equiv 3\pmod 4$ and $|\Delta|\in \{q_1, q_2, 4q_1, 4q_2, q_1q_2, 4q_1q_2\}$, or $p\equiv 1\pmod 4$ and $|\Delta|\in \{q_2, 4q_2\}$.
		 	
		 	\item $p\equiv 3\pmod 4$, $q_1\equiv 3\pmod 4$,  $\ord_{q_1^{r_1}}p=q_1^{r_1-1}(q_1-1)/2$, $\ord_{q_2^{r_2}}p=q_2^{r_2-1}(q_2-1)/2$,  $\gcd(q_1^{r_1-1}(q_1-1)/2, q_2^{r_2-1}(q_2-1)/2)=1$,  and
		 	$d=\Delta\in \{-q_1, -q_2, q_1q_2\}$.
		 \end{itemize}

%
%
%
%
%
%
%
		
		 \item $\exp(G)=8q^r$ and one of the following holds

		 \begin{itemize}
		 	\item  $p$ is a primitive root of $q^r$ and either  $p\equiv 1\pmod 8$,  $d=\pm 2$,  and $\Delta=4d$, or  $p\not\equiv 1\pmod 8$ and  $|\Delta|\in \{q, 4q, 8\}$.

		 	\item  $q\equiv 3\pmod 4$,  $\ord_{q^r}p=q^{r-1}(q-1)/2$, and either $p\equiv 5\pmod 8$ and $\Delta\in \{4q, -q\}$, or $p\not\equiv 1\pmod 4$   and	$\Delta\in \{-q, 4q, \pm 8\}$.
		 \end{itemize}

		 \item $\exp(G)=2^tq^r$ with $t\ge 4$ and $q\equiv 3\pmod 4$, $p\equiv \pm3,\pm5\pmod {16}$,  $|\Delta|\in \{8, q, 4q\}$, and  either $p$ is a primitive root of $q^r$, or  $\ord_{q^r}p=q^{r-1}(q-1)/2$.

	\end{enumerate}

\item $\mathcal O_{\mathfrak p}[G]$ is weekly clean but not clean  if and only if
one of the following holds.
\begin{enumerate}

	\item $\exp(G)=2^r$ with $r\ge 3$, $|\Delta|=8$, $d=\pm 2$, and $p\equiv 5\pmod {8}$,
	$G\cong C_{2^r}$, $x^2+1$ is irreducible over $\mathcal O_{\mathfrak p}$

%
		\item $\exp(G)=4q^r$, $|\Delta|=4q$, $d=\pm q\equiv 3\pmod 4$, $p\equiv 1\pmod 4$, and $p$ is a primitive root of $q^{\epsilon(r)}$, $G\cong C_{4}\oplus H$, where $H$ is a $q$-group, $x^2+1$ is irreducible over $\mathcal O_{\mathfrak p}$.

			\item $\exp(G)=q_1^{r_1}q_2$ or $2q_1^{r_1}q_2$, $|\Delta|=q_1q_2$, $d=\pm q_1q_2\equiv 1\pmod 4$, $p$ is a primitive root of $q_1^{\epsilon(r_1)}$, $\ord_{q_2}p=(q_2-1)/2$, and $\gcd(q_1^{\epsilon(r_1)-1}(q_1-1), (q_2-1))=2$, $G\cong C_{q_2}\oplus H$, where $H$ is a $q_1$-group, $\Phi_{q_2}(x)$ is irreducible over $\mathcal O_{\mathfrak p}$.

	\item $G\cong C_{\exp(G)}$, $\Phi_{\exp(G)}(x)$ is irreducible over $\mathcal O_{\mathfrak p}$, and  one of the following holds
	 \begin{itemize}
		\item $\exp(G)=q$ with $q\equiv 1\pmod 4$, $\Delta=d=q$, and $\ord_{q}p=(q-1)/4$.
		
		\item $\exp(G)=q_1q_2$, $\Delta=d=q_1q_2$, $q_1\equiv 1\pmod 4$, $q_2\equiv 1\pmod 4$, $p$ is a primitive root of both $q_1$ and $q_2$, and $\gcd(q_1-1, q_2-1)=4$.
		
		\item $\exp(G)=8$, $|\Delta|=8$, $d=\pm2$, and $p\equiv 7,15\pmod {16}$.
		
		\item $\exp(G)=q_1q_2$, $|\Delta|=q_1$, and $d=\pm q_1\equiv 1\pmod 4$, where $q_1,q_2$ are distinct odd primes such that $p$ is a primitive root of $q_2$, $\ord_{q_1}p=(q_1-1)/2$, and $\gcd((q_1-1)/2, q_2-1)=2$.
		
		\item $\exp(G)=4q$, $|\Delta|=q=d\equiv 1\pmod 4$,  and $p\equiv 3\pmod 4$, where $q$ is an odd prime such that $\ord_qp=(q-1)/2$.
		
	\end{itemize}
\end{enumerate}

\end{enumerate}

\item  Suppose  $\Delta\mid \exp(G)$ and $\big(\frac{\Delta}{p}\big)=-1$. Then

\begin{enumerate}
	\item $\mathcal O_{\mathfrak p}[G]$ is  clean  if and only if $\mathcal O_{\mathsf p}[\zeta_m]$ is clean for all divisors $m$ of $n$ if and only if $\exp(G)=q^r$ or $2q^r$, $|\Delta|=q$,   $d=\pm q\equiv 1\pmod 4$,  and either $p$ is a primitive root of $q^{\epsilon}$, or $q\equiv 3\pmod 4$ and $\ord_{q^{\epsilon(r)}}p=q^{\epsilon(r)-1}(q-1)/2$.

\item $\mathcal O_{\mathfrak p}[G]$ is feebly clean  if and only if $\mathcal O_{\mathsf p}[\zeta_{m}]$ is feebly clean for all divisors $m$ of $\exp(G)$ if and only if $\ord_{\exp(G)}p=\varphi(\exp(G))$ with $\exp(G)\ge 3$, or
	$\ord_{\exp(G)}p=\varphi(\exp(G))/2$ is odd, or one of the following holds
	\begin{enumerate}
		\item  $\exp(G)=8$, $|\Delta|=8$, and  $d=\pm2$.

		\item $\exp(G)=2^r$ with $r\ge4$, $|\Delta|=8$, $d=\pm2$, and $p\equiv \pm3,\pm5\pmod {16}$.

		\item $\exp(G)=q^r$ or $2q^r$, where $q$ is an odd prime with $q\equiv 1\pmod 4$, $\Delta=d=q$,  and  $\ord_{q^{\epsilon(r)}}p=q^{\epsilon(r)-1}(q-1)/2$.

		\item $\exp(G)=q^r$ or $2q^r$, where $q$ is an odd prime with $q\equiv 5\pmod 8$, $\Delta=d=q$,  and  $\ord_{q^{\epsilon(r)}}p=q^{\epsilon(r)-1}(q-1)/4$.

		\item $\exp(G)=4q^r$, where $q$ is an odd prime and $r\in \N$, $|\Delta|=q$, $d=\pm q\equiv 1\pmod 4$, and  either  $p$ is a primitive root of $q^{\epsilon(r)}$, or $p\equiv 3\pmod 4$, $q\equiv 3\pmod 4$, and $\ord_{q^{\epsilon(r)}}p=q^{\epsilon(r)-1}(q-1)/2$.

		 \item $\exp(G)=4q^r$, where $q$ is an odd prime and $r\in \N$, $|\Delta|=4q$, $d=\pm q\equiv 3\pmod 4$, and  either  $p$ is a primitive root of $q^{\epsilon(r)}$, or $p\equiv 3\pmod 4$, $q\equiv 3\pmod 4$, and $\ord_{q^{\epsilon(r)}}p=q^{\epsilon(r)-1}(q-1)/2$.

		\item $\exp(G)=4q^r$, where $q$ is an odd prime and $r\in \N$, $\Delta=d=-q\equiv 1\pmod 4$,   $p\equiv 1\pmod 4$,  and $\ord_{q^{\epsilon(r)}}p=q^{\epsilon(r)-1}(q-1)/2$.

		\item $\exp(G)=4q^r$, where $q$ is an odd prime and $r\in \N$, $\Delta=4q$, $d= q\equiv 3\pmod 4$,  $p\equiv 1\pmod 4$, $\ord_{q^{\epsilon(r)}}p=q^{\epsilon(r)-1}(q-1)/2$.

		\item $\exp(G)=q_1^{r_1}q_2^{r_2}$ or $2q_1^{r_1}q_2^{r_2}$, where $q_1,q_2$ are distinct odd primes with $q_2\equiv 3\pmod 4$ and $\gcd(q_1^{\epsilon(r_1)-1}(q_1-1), q_2^{\epsilon(r_2)-1}(q_2-1))=2$, $|\Delta|=|d|=\{q_1, q_2, q_1q_2\}$ with $d\equiv 1\pmod 4$,
		 and either $p$ is a primitive root of both $q_1^{\epsilon(r_1)},q_2^{\epsilon(r_2)}$, or $p$ is a primitive root of $q_1^{\epsilon(r_1)}$ and $\ord_{q_2^{\epsilon(r_2)}}p=q_2^{\epsilon(r_2)-1}(q_2-1)/2$.

%
%

		 \item $\exp(G)=q_1^{r_1}q_2^{r_2}$ or $2q_1^{r_1}q_2^{r_2}$, where $q_1,q_2$ are distinct  odd primes with $q_1\equiv 3\pmod 4$ and $q_2\equiv 3\pmod 4$,  and either $p$ is a primitive root of $q_1^{\epsilon(r_1)}$, $\ord_{q_2^{r_2}}p=q_2^{r_2-1}(q_2-1)/2$, and $\gcd(q_1^{r_1-1}(q_1-1), q_2^{r_2-1}(q_2-1)/2)=1$,  or $\ord_{q_1^{r_1}}p=q_1^{r_1-1}(q_1-1)/2$, $\ord_{q_2^{r_2}}p=q_2^{r_2-1}(q_2-1)/2$, and $\gcd(q_1^{r_1-1}(q_1-1)/2, q_2^{r_2-1}(q_2-1)/2)=1$, then $|\Delta|=|d|\in \{q_1,q_2,q_1q_2\}$ and $d\equiv 1\pmod 4$.
	\end{enumerate}
	
	\item $\mathcal O_{\mathfrak p}[G]$ is weekly clean but not clean  if and only if
  $G\cong C_q$, $\Phi_q(x)$ is irreducible over $\mathcal O_{\mathfrak p}$,
		$|\Delta|=q$,   $d= q\equiv 1\pmod 4$,  and either $\ord_qp=(q-1)/4$ is odd, or $\ord_qp=(q-1)/2$, where $q$ is an odd prime.
\end{enumerate}

 \end{enumerate}

\end{theorem}

\begin{proof}
Since $p\neq 2$ and $p\nmid \exp(G)$, we have $\{2,m\}\subset  U(\mathcal O_{\mathfrak p})$ for all divisors $m$ of $\exp(G)$.

\smallskip
\textbf{(1)} Suppose $\Delta\nmid \exp(G)$. By Lemma \ref{3.4} we have $\Q(\sqrt{d})\not\subset \Q(\zeta_{\exp(G)})$. Let $m$ be a divisor of $\exp(G)$.
Then $\deg(\phi_m(x))=[\Q(\sqrt{d})(\zeta_m):\Q(\sqrt{d})]=\varphi(m)$.

\smallskip
\textbf{(1.a)} Suppose $\mathcal O_{\mathsf p}[G]$ is clean. Then Lemma \ref{lemm1} and Theorem \ref{t1} imply that $\mathcal O_{\mathsf p}[\zeta_{\exp(G)}]$ is clean.

Suppose $\mathcal O_{\mathsf p}[\zeta_{\exp(G)}]$ is clean. Then by Proposition \ref{2.5}.1 $\deg(\phi_{\exp(G)}(x))=\ord_{\exp(G)}N(\mathsf p)$. If $\big(\frac{\Delta}{p}\big)=1$ or $0$, then by Lemma \ref{3.5} $N(\mathsf p)=p$ and hence $\varphi(\exp(G))=\ord_{\exp(G)}p$, i.e., (i) holds. If $\big(\frac{\Delta}{p}\big)=-1$, then by Lemma \ref{3.5} $N(\mathsf p)=p^2$ and hence $\varphi(\exp(G))=\ord_{\exp(G)}p^2=\frac{\ord_{\exp(G)}p}{\gcd(2, \ord_{\exp(G)}p)}\le \ord_{\exp(G)}p\le \varphi(\exp(G))$. Thus $\ord_{\exp(G)}p=\varphi(\exp(G))$ is odd, i.e., $\exp(G)= 2$ and hence (ii)  holds.

Suppose (i) or (ii) holds. If (i) holds, then $N(\mathsf p)=p$ and for every divisor $m$ of $\exp(G)$, we have $p$ is a primitive root of $m$, whence $\deg(\phi_m(x))=\varphi(m)=\ord_mp=\ord_mN(\mathsf p)$ for all divisors $m$ of $\exp(G)$. It follows from Proposition \ref{2.5}.1 that $\mathcal O_{\mathsf p}[G]$ is clean.
If (ii) holds, then $\exp(G)=2$. Since $\mathcal O_{\mathsf p}[\zeta_{\exp(G)}]=\mathcal O_{\mathsf p}$ is clean, by Theorem \ref{t1} we obtain  that  $\mathcal O_{\mathsf p}[G]$ is clean.

\smallskip
\textbf{(1.b)} Suppose $\mathcal O_{\mathsf p}[G]$ is feebly clean but not clean. Then by (1) $\mathcal O_{\mathsf p}[\zeta_{\exp(G)}]$ is not clean, 
and by Lemma \ref{lemm1} and Theorem \ref{t1},  $\mathcal O_{\mathsf p}[\zeta_{\exp(G)}]$ is feebly clean.

 Suppose $\mathcal O_{\mathsf p}[\zeta_{\exp(G)}]$ is feebly clean but not clean. It follows from Proposition \ref{2.5}.2 that $\deg(\phi_{\exp(G)}(x))=2\ord_{\exp(G)}N(\mathsf p)$. If $\big(\frac{\Delta}{p}\big)=1$ or $0$, then by Lemma \ref{3.5} $N(\mathsf p)=p$ and hence $\varphi(\exp(G))=2\ord_{\exp(G)}p$, i.e., (i) holds. Suppose $\big(\frac{\Delta}{p}\big)=-1$. Then by Lemma \ref{3.5} $N(\mathsf p)=p^2$ and hence $\varphi(\exp(G))=2\ord_{\exp(G)}p^2=2\frac{\ord_{\exp(G)}p}{\gcd(2, \ord_{\exp(G)}p)}$.
 If $\gcd(2, \ord_{\exp(G)}p)=1$, then $\varphi(\exp(G))/2=\ord_{\exp(G)}p$ is odd, i.e., (ii) holds.
 If $\gcd(2, \ord_{\exp(G)}p)=2$, then $\varphi(\exp(G))=\ord_{\exp(G)}p$ is even, i.e., (iii) holds.

 Suppose (i), or (ii), or (iii) holds. If (i) holds, then $N(\mathsf p)=p$ and for every divisor $m$ of $\exp(G)$, we have $\varphi(m)\ge 2\ord_mp$, whence $\deg(\phi_m(x))=\varphi(m)\ge 2\ord_mp=2\ord_mN(\mathsf p)$ for all divisors $m$ of $\exp(G)$. It follows from Proposition \ref{2.5} that $\mathcal O_{\mathsf p}[G]$ is feebly clean. By (1.a.i), we have $\mathcal O_{\mathsf p}[G]$ is not clean.
 If (ii) or (iii) holds, then $\exp(G)\neq 2$, $N(\mathsf p)=p^2$ and $\varphi(\exp(G))=2\frac{\ord_{\exp(G)}p}{\gcd(2, \ord_{\exp(G)}p)}=2\ord_{\exp(G)}p^2$.
 Thus for every divisor $m$ of $\exp(G)$, we have $\varphi(m)\ge 2\ord_mp^2$, whence $\deg(\phi_m(x))=\varphi(m)\ge 2\ord_mN(\mathsf p)$ for all divisors $m$ of $\exp(G)$. It follows from Proposition \ref{2.5} that $\mathcal O_{\mathsf p}[G]$ is feebly clean. By (1.a.ii) and $\exp(G)\neq 2$, we have $\mathcal O_{\mathsf p}[G]$ is not clean.

\smallskip
\textbf{(1.c)} Suppose $\mathcal O_{\mathsf p}[G]$ is weakly clean but not clean. Then by (1.a), we have $\mathcal O_{\mathsf p}[\zeta_{\exp(G)}]$ is not clean. Assume to the contrary that $G$ is not cyclic or $\Phi_{\exp(G)}(x)$ is not irreducible over $\mathcal O_{\mathsf p}$. Then $\lambda(\exp(G))\ge 2$. In view of Theorem \ref{t1} and Lemma \ref{lemm1}, we have  $\mathcal O_{\mathsf p}[G]$ is not weakly clean, a contradiction. Thus $G\cong C_{\exp(G)}$ and $\Phi_{\exp(G)}(x)$ is  irreducible over $\mathcal O_{\mathsf p}$. It  follows from Theorem \ref{t1} and Lemma \ref{lemm1} that $\mathcal O_{\mathsf p}[\zeta_m]$ is clean for all divisors $m$ of $\exp(G)$ with $m\neq \exp(G)$ and $\mathcal O_{\mathsf p}[\zeta_{\exp(G)}]$ is feebly clean but not clean.

 If $\big(\frac{\Delta}{p}\big)=-1$, then  by Lemma \ref{3.5} $N(\mathsf p)=p^2$ and
in view of Proposition \ref{2.5}, we obtain $\deg(\phi_{\exp(G)}(x))=\varphi(\exp(G))=2\ord_{\exp(G)}(p^2)$ and for all divisors $m$ of $\exp(G)$ with $m\neq \exp(G)$, $\deg(\phi_m(x))=\varphi(m)=\ord_mp^2$. Therefore $\ord_mp=\varphi(m)$ is odd, which implies that $m=1$ or $2$. Thus $\exp(G)=4$.

If  $\big(\frac{\Delta}{p}\big)\neq -1$, then  by Lemma \ref{3.5} $N(\mathsf p)=p$ and
in view of Proposition \ref{2.5}, we obtain $\deg(\phi_{\exp(G)}(x))=\varphi(\exp(G))=2\ord_{\exp(G)}p$ and for all divisors $m$ of $\exp(G)$ with $m\neq \exp(G)$, $\deg(\phi_m(x))=\varphi(m)=\ord_mp$. The assertion follows from Proposition \ref{p3.2}.2.

Suppose $G\cong C_{\exp(G)}$,   $\Phi_{\exp(G)}(x)$ is  irreducible over $\mathcal O_{\mathsf p}$, and either  $\exp(G)=4$ and $\big(\frac{\Delta}{p}\big)=-1$, or  $\big(\frac{\Delta}{p}\big)\neq -1$ and one of (i)-(iv) holds. Then $\lambda(\exp(G))=1$ and by Proposition \ref{p3.2}.2 we have $\deg(\phi_{\exp(G)}(x))=\varphi(\exp(G))=2\ord_{\exp(G)}(N(\mathsf p))$ and for all divisors $m$ of $\exp(G)$ with $m\neq \exp(G)$, $\deg(\phi_m(x))=\varphi(m)=\ord_mN(\mathsf p)$. It follows from Proposition \ref{2.5} that $\mathcal O_{\mathsf p}[\zeta_m]$ is clean for all divisors $m$ of $\exp(G)$ with $m\neq \exp(G)$ and $\mathcal O_{\mathsf p}[\zeta_{\exp(G)}]$ is feebly clean but not clean. In view of Theorem \ref{t1} and Lemma \ref{lemm1}, we have $\mathcal O_{\mathsf p}[G]$ is weakly clean but not clean.

\smallskip
\textbf{(2)}  Suppose $\Delta\mid \exp(G)$ and  $\big(\frac{\Delta}{p}\big)\neq -1$. Then $N(\mathsf p)=p$.
 By Lemma \ref{3.4} we have $\Q(\sqrt{d})\subset \Q(\zeta_{\exp(G)})$. Let $m$ be a divisor of $\exp(G)$ with $\Delta\nmid m$.
Then $\deg(\phi_m(x))=[\Q(\sqrt{d})(\zeta_m):\Q(\sqrt{d})]=\varphi(m)$. Let $m$ be a divisor of $\exp(G)$ with $\Delta\mid m$.
Then $\deg(\phi_m(x))=[\Q(\sqrt{d})(\zeta_m):\Q(\sqrt{d})]=\varphi(m)/2$.

\smallskip
 \textbf{(2.a)} It is clear that   $\mathcal O_{\mathfrak p}[G]$ is clean if and only if $\mathcal O_{\mathfrak p}[\zeta_{m}]$ is clean for all divisors $m$ of $\exp(G)$. By a careful check, it is easy to see that each  of  the cases (i)-(vii) implies that $\mathcal O_{\mathsf p}[\zeta_m]$ is clean for all divisors $m$ of $n$ and hence  $\mathcal O_{\mathfrak p}[G]$ is clean. Thus we only need to show that the cleanness of $\mathcal O_{\mathfrak p}[G]$  implies that one of the cases (i)-(vii) holds.

  Now suppose $\mathcal O_{\mathfrak p}[\zeta_{m}]$ is clean for all divisors $m$ of $\exp(G)$.
  Then
   $\varphi(|\Delta|)/2=\ord_{|\Delta|}p$ and $\varphi(m)=\ord_mp$ for divisor $m$ of $|\Delta|$ with $m\neq |\Delta|$. Note that $\Delta=d$ or $4d$. Thus by Proposition \ref{p3.2}.2  one of the following holds
 \begin{itemize}
 	\item  $|\Delta|=8$ and $p\equiv 3\pmod 4$.

 	\item $|\Delta|=q$ is an odd prime  and  $\ord_{q}p=(q-1)/2$.

 	\item $|\Delta|=4q$, where $q$ is an odd prime, $p\equiv 3\pmod 4$,  and    $p$ is a primitive root of $q$.

 	\item $|\Delta|=q_1q_2$ is a product of  two distinct primes, $p$ is a primitive root of unity of both $q_1$ and $q_2$, and $\gcd(q_1-1,q_2-1)=2$.
 \end{itemize}
Since $\mathcal O_{\mathfrak p}[\zeta_{\exp(G)}]$ is clean, we have $\ord_{\exp(G)}p=\varphi(\exp(G))/2$. Next we check all the cases of Proposition \ref{p2.6}.2.

Suppose $\exp(G)=2^r$ with $r\in \N$. Then $|d|=2$ and $|\Delta|=8$, whence $r\ge 3$. If $r=3$, this is the case (i). If $r\ge 4$, then  $p\equiv \pm3,\pm5\pmod {16}$, and this  together with $p\equiv 3\pmod 4$ gives $p\equiv 3, 11\pmod {16}$, which is the case (ii).

Suppose $\exp(G)=q^r$ or $2q^r$, where $r\ge 1$ and $q$ is an odd prime with $\ord_{q^{\epsilon(r)}}p=q^{\epsilon(r)-1}(q-1)/2$. Then $|\Delta|=q$ and $d=\pm q\equiv 1\pmod 4$. This is the case (iii).

Suppose $\exp(G)=4q^r$, where  $r\ge 1$ and $q$ is an odd prime such that $p$ is a primitive root of $q^{\epsilon(r)}$. Then $|\Delta|=4q$, $d=\pm q\equiv 3\pmod 4$, and $p\equiv 3\pmod 4$. This is the case (v).

Suppose $\exp(G)=4q^r$, where  $r\ge 1$ and $q$ is an odd prime with $p\equiv 3\pmod 4$, $q\equiv 3\pmod 4$, and  $\ord_{q^{\epsilon(r)}}=q^{\epsilon(r)-1}(q-1)/2$. Then $|\Delta|=q$ and $d=- q\equiv 1\pmod 4$. This is the case (iv).

Suppose $\exp(G)=q_1^{r_1}q_2^{r_2}$ or $2q_1^{r_1}q_2^{r_2}$, where $q_1,q_2$ are distinct odd primes,  $\gcd(q_1^{\epsilon(r_1)-1}(q_1-1), q_2^{\epsilon(r_2)-1}(q_2-1))=2$,  and  $p$ is a primitive root of both $q_1^{\epsilon(r_1)},q_2^{\epsilon(r_2)}$. Then $|\Delta|=q_1q_2$ and $d=\pm q_1q_2\equiv 1\pmod 4$. This is the case of (vii).

Suppose $\exp(G)=q_1^{r_1}q_2^{r_2}$ or $2q_1^{r_1}q_2^{r_2}$, where $q_1,q_2$ are distinct odd primes with $q_2\equiv 3\pmod 4$,  $\gcd(q_1^{\epsilon(r_1)-1}(q_1-1), q_2^{\epsilon(r_2)-1}(q_2-1))=2$,  $p$ is a primitive root of $q_1^{\epsilon(r_1)}$, and $\ord_{q_2^{\epsilon(r_2)}}p=q_2^{\epsilon(r_2)-1}(q_2-1)/2$.
Then $|\Delta|=q_2$ and $d=-q_2\equiv 1\pmod 4$. This is the case (vi).

\smallskip
\textbf{(2.b)}  It is clear that $\mathcal O_{\mathfrak p}[G]$ is feebly clean if and only if $\mathcal O_{\mathfrak p}[\zeta_{m}]$ is feebly clean for all divisors $m$ of $\exp(G)$.

Since $\mathcal O_{\mathfrak p}[\zeta_{\exp(G)}]$ being feebly clean implies that  $\varphi(\exp(G))/2=\ord_{\exp(G)}p$  or  $\varphi(\exp(G))/2=2\ord_{\exp(G)}p$, we distinguish two cases.

\smallskip
\noindent{\bf Case 1.} $\varphi(\exp(G))/2=\ord_{\exp(G)}p$.
 \smallskip

In this case, we only need to show that $\mathcal O_{\mathfrak p}[\zeta_{m}]$ is feebly clean for all divisors $m$ of $\exp(G)$ and this is clear by the fact that $\ord_mp\ge \varphi(m)/2$ for all divisors $m$ of $\exp(G)$.

\smallskip
\noindent{\bf Case 2.} $\varphi(\exp(G))/2=2\ord_{\exp(G)}p$.
\smallskip

If one of the cases (i)-(ix) holds, then a careful check implies that $\mathcal O_{\mathfrak p}[\zeta_{m}]$ is feebly clean for all divisors $m$ of $\exp(G)$ and hence $\mathcal O_{\mathfrak p}[G]$ is feebly clean. We only need to show that $\mathcal O_{\mathfrak p}[G]$ is feebly clean implies that one of the cases (i)-(ix) holds.

Now suppose $\mathcal O_{\mathfrak p}[\zeta_{m}]$ is feebly clean for all divisors $m$ of $\exp(G)$.
Let $m$ be a divisor of $\exp(G)$. If $\Delta\mid m$, then $\ord_{m}p\ge \varphi(m)/4$ and if $\Delta\nmid m$, then $\ord_mp\ge \varphi(m)/2$. In particular, if there exists a divisor $m$ of $\exp(G)$ such  that $\ord_mp=\varphi(m)/4$, then $\Delta\mid m$.  Note that $d$ is square-free.  We will use all those facts without further mention.


Since $\varphi(\exp(G))/2=2\ord_{\exp(G)}p$, we check all the cases of Proposition \ref{p2.6}.3 for $\exp(G)$.

If $\exp(G)=8$, then $|\Delta|=8$, $d=\pm 2$, and $p\equiv 1\pmod 8$. This is the case (i).

If $\exp(G)=2^r$ with $r\ge 4$, then $|\Delta|=8$ and  $d=\pm 2$. Since $\ord_{2^r}p=2^{r-3}$, we obtain $\ord_{16}p=2$ and $\ord_8p=1$. Together with Proposition \ref{p2.6}.2.(a) and (b), we have  $p\equiv 9\pmod {16}$. This is the case (ii).

If $\exp(G)=q^r$ or $2q^r$ with $q\equiv 1\pmod 4$ and  $\ord_{q^{\epsilon(r)}}p=q^{\epsilon(r)-1}(q-1)/4$, then   $\Delta=d=q$. This is the case (iii).

 If $\exp(G)=q_1^{r_1}q_2^{r_2}$ or $2q_1^{r_1}q_2^{r_2}$, Then  $|\Delta|$ divides $q_1q_2$. By comparing with Proposition \ref{p2.6}.3(d) and by symmetry we obtain that  one of the following holds, which is the case (iv).

\begin{itemize}
	\item $p$ is a primitive root of both $q_1^{\epsilon(r_1)}, q_2^{\epsilon(r_2)}$, $\gcd(q_1^{\epsilon(r_1)-1}(q_1-1), q_2^{\epsilon(r_2)-1}(q_2-1))=4$, and $|\Delta|\in \{q_1, q_2, q_1q_2\}$. Note that $q_1\equiv 1\pmod 4$ and $q_2\equiv 1\pmod 4$.  Thus  $\Delta=d\in \{q_1, q_2, q_1q_2\}$.
	
	\item $p$ is a primitive root of $q_1^{\epsilon(r_1)}$, $\ord_{q_2^{\epsilon(r_2)}}p=q_2^{\epsilon(r_2)-1}(q_2-1)/2$,  $\gcd(q_1^{\epsilon(r_1)-1}(q_1-1), q_2^{\epsilon(r_2)-1}(q_2-1)/2)=2$, $|\Delta|\in \{q_1, q_2, q_1q_2\}$, and $d=\Delta\equiv 1\pmod 4$.

	\item $\ord_{q_1^{\epsilon(r_1)}}p=q_1^{\epsilon(r_1)-1}(q_1-1)/2$, $\ord_{q_2^{\epsilon(r_2)}}p=q_2^{\epsilon(r_2)-1}(q_2-1)/2$, $\gcd(q_1^{\epsilon(r_1)-1}(q_1-1)/2, q_2^{\epsilon(r_2)-1}(q_2-1)/2)=1$,  $|\Delta|\in \{q_1, q_2, q_1q_2\}$, and  $d=\Delta\equiv 1\pmod 4$.

	\item $p$ is a primitive root of $q_1^{\epsilon(r_1)}$, $\ord_{q_2^{\epsilon(r_2)}}p=q_2^{\epsilon(r_2)-1}(q_2-1)/4$,  $\gcd(q_1^{\epsilon(r_1)-1}(q_1-1), q_2^{\epsilon(r_2)-1}(q_2-1)/4)=1$, and $|\Delta|=q_2$. Since $q_2\equiv 1\pmod 4$, we obtain $d=\Delta=q_2$.
\end{itemize}

If $\exp(G)=q_1^{r_1}q_2^{r_2}q_3^{r_3}$ or $2q_1^{r_1}q_2^{r_2}q_3^{r_3}$, then  $|\Delta|$ divides $q_1q_2q_3$. By comparing with Proposition \ref{p2.6}.3(e) and by symmetry  we obtain that  one of the following holds, which is the case (v).
\begin{itemize}
	\item $p$ is a primitive root of  $q_1^{r_1}, q_2^{r_2}, q_3^{r_3}$,  $\gcd(q_i^{r_i-1}(q_i-1), q_j^{r_j-1}(q_j-1))=2$  for all distinct $i,j\in \{1,2,3\}$, $|\Delta|\in \{q_1, q_2, q_3, q_1q_2, q_1q_3, q_2q_3, q_1q_2q_3\}$, and $d=\Delta\equiv 1\pmod 4$.
	
	\item $p$ is a primitive root of $q_1^{r_1}, q_2^{r_2}$, $\ord_{q_3^{r_3}}p=q_3^{r_3-1}(q_3-1)/2$ is odd,   $\gcd(q_i^{r_i-1}(q_i-1), q_j^{r_j-1}(q_j-1))=2$ for all distinct $i,j\in \{1,2,3\}$, $|\Delta|\in \{q_1, q_2, q_3, q_1q_2, q_1q_3, q_2q_3, q_1q_2q_3\}$, and $d=\Delta\equiv 1\pmod 4$.

	\item $p$ is a primitive root of $q_1^{r_1}$, $\ord_{q_2^{r_2}}p=q_2^{r_2-1}(q_2-1)/2$ is odd, $\ord_{q_3^{r_3}}p=q_3^{r_3-1}(q_3-1)/2$ is odd, and $\gcd(q_i^{r_i-1}(q_i-1), q_j^{r_j-1}(q_j-1))=2$ for all distinct $i,j\in \{1,2,3\}$. Note that $\ord_{q_2q_3}p=\varphi(q_2q_3)/4$. We obtain that $|\Delta|\mid q_2q_3$. Since $q_2\equiv 3\pmod 4$ and $q_3\equiv 3\pmod 4$, we obtain  $d=\Delta\in \{-q_2, -q_3, q_2q_3\}$.
\end{itemize}


 If $\exp(G)=4q^r$, then $|\Delta|=q$ or $4q$. By comparing with Proposition \ref{p2.6}.3(f) and by symmetry, we obtain that  one of the following holds,  which is the case (vi).

 \begin{itemize}
 	\item $p\equiv 1\pmod 4$, $\ord_{q^r}p=q^{r-1}(q-1)/2$, and either $d=\Delta=\pm q\equiv 1\pmod 4$, or $d=\Delta/4=\pm q\equiv 3\pmod 4$.
 	
 	\item $p\equiv 3\pmod 4$, $q\equiv 1\pmod 4$,   $\ord_{q^r}p=q^{r-1}(q-1)/2$, and  either $d=\Delta=q$, or $d=\Delta/4=-q$.
 	
 	\item $p\equiv 3\pmod 4$, $q\equiv 5\pmod 8$, and $\ord_{q^r}p=q^{r-1}(q-1)/4$. Then $|\Delta|=q$ and hence $d=\Delta=q$.
 \end{itemize}

 If $\exp(G)=4q_1^{r_1}q_2^{r_2}$, then  $|\Delta|$ divides $4q_1q_2$. By comparing with Proposition \ref{p2.6}.3(g) and by symmetry  we can assume $q_2\equiv 3\pmod 4$  and  we obtain that one of the following holds,  which is the case (vii).

 \begin{itemize}
 	
 	\item  $\gcd(q_1^{\epsilon(r_1)-1}(q_1-1), q_2^{\epsilon(r_2)-1}(q_2-1))=2$,  $p$ is a primitive root of both $q_1^{\epsilon(r_1)},q_2^{\epsilon(r_2)}$,  $|\Delta|\in \{q_1, q_2, 4q_1, 4q_2, q_1q_2, 4q_1q_2\}$, and either $d=\Delta\equiv 1\pmod 4$, or $d=\Delta/4\equiv 3\pmod 4$.

 	\item   $\gcd(q_1^{\epsilon(r_1)-1}(q_1-1), q_2^{\epsilon(r_2)-1}(q_2-1))=2$,  $p$ is a primitive root of $q_1^{\epsilon(r_1)}$, and $\ord_{q_2^{\epsilon(r_2)}}p=q_2^{\epsilon(r_2)-1}(q_2-1)/2$.  Note that if $p\equiv 1\pmod 4$, then  $\ord_{4q_2}p=(q_2-1)/2=\varphi(4q_2)/4$ and hence $|\Delta|\mid 4q_2$. Therefore
 	either $p\equiv 3\pmod 4$ and $|\Delta|\in \{q_1, q_2, 4q_1, 4q_2, q_1q_2, 4q_1q_2\}$, or $p\equiv 1\pmod 4$ and $|\Delta|\in \{q_2, 4q_2\}$.
 	
 	\item $p\equiv 3\pmod 4$, $q_1\equiv 3\pmod 4$,  $\ord_{q_1^{r_1}}p=q_1^{r_1-1}(q_1-1)/2$, $\ord_{q_2^{r_2}}p=q_2^{r_2-1}(q_2-1)/2$, and $\gcd(q_1^{r_1-1}(q_1-1)/2, q_2^{r_2-1}(q_2-1)/2)=1$.  Note that  $\ord_{q_1q_2}p=\varphi(q_1q_2)/4$ and hence $|\Delta|\mid q_1q_2$. Therefore
 	 $d=\Delta\in \{-q_1, -q_2, q_1q_2\}$.
 \end{itemize}

%
%

If $\exp(G)=8q^r$, then $|\Delta|=q$ or $4q$ or $8$. By comparing with Proposition \ref{p2.6}.3(h) and by symmetry we obtain that  one of the following holds, which is the case (viii).
\begin{itemize}
	\item  $p$ is a primitive root of $q^r$. If $\ord_8p=1$, i.e., $p\equiv 1\pmod 8$, then $|\Delta|\mid 8$ and hence $d=\pm 2$ and $\Delta=4d$. Otherwise $|\Delta|\in \{q, 4q, 8\}$ and either $d=\Delta\equiv 1\pmod 4$ or $d=\Delta/4\not\equiv 1\pmod 4$.

	\item $p\not\equiv 1\pmod 8$, $q\equiv 3\pmod 4$, and $\ord_{q^r}p=q^{r-1}(q-1)/2$.	If $p\equiv 5\pmod 8$, then $\ord_{4q}=(q-1)/2=\varphi(4q)/4$ and hence $|\Delta|\mid 4q$. Therefore either $p\equiv 5\pmod 8$ and $\Delta\in \{4q, -q\}$, or $p\not\equiv 1\pmod 4$ and	$\Delta\in \{-q, 4q, \pm 8\}$.
\end{itemize}


 If $\exp(G)=2^tq^r$ with $t\ge 4$, then $q\equiv 3\pmod 4$, $p\equiv \pm3,\pm5\pmod {16}$,  $|\Delta|\in \{8, q, 4q\}$, and  either $p$ is a primitive root of $q^r$, or  $\ord_{q^r}p=q^{r-1}(q-1)/2$.  This is the case (ix).


\smallskip
\textbf{(2.c)} It is easy to check that each case (i)-(v) implies that $\mathcal O_{\mathfrak p}[G]$ is weakly clean but not clean. Now we prove the converse and
suppose $\mathcal O_{\mathfrak p}[G]$ is weakly clean but not clean. It follows from Lemma \ref{lemm2} and Theorem \ref{t1} that there exists exactly one
 divisor $b$ of $\exp(G)$ such that $\mathcal O_{\mathfrak p}[\zeta_{b}]$ is not clean. Then $b\ge 3$. We distinguish two cases.

\smallskip
\noindent{\bf Case 1. }$\Delta\nmid b$.
\smallskip

If there exists a divisor $b'$ of $\exp(G)$ such that $\Delta\nmid b'$ and  $b\t b'$ with $b\neq b'$, then
$\mathcal O_{\mathfrak p}[\zeta_{b'}]$ is  clean and hence $\mathcal O_{\mathfrak p}[\zeta_{b}]$   clean, a contradiction. Thus
 $b$ is  a maximal divisor of $\exp(G)$ such that $\Delta\nmid b$. For every divisor $b_1$ of $b$ with $b_1\neq b$, we have $\mathcal O_{\mathfrak p}[\zeta_{b_1}]$ is clean and $\varphi(b_1)=\ord_{b_1}p$. By Proposition \ref{p3.2}.2, one of the following holds.

 \begin{itemize}
 	\item $b=4$ and $p\equiv 1\pmod 4$; or $b=8$ and $p\equiv 3\pmod 4$.

 	\item $b=q$ is an odd prime  and  $\ord_{q}p=(q-1)/2$.

 	\item $b=4q$, where $q$ is an odd prime, $p\equiv 3\pmod 4$, and    $p$ is a primitive root of $q$.

 	\item $b=q_1q_2$ is a product of  two distinct primes, $p$ is a primitive root of unity of both $q_1$ and $q_2$, and $\gcd(q_1-1,q_2-1)=2$.
 \end{itemize}

 Since $\mathcal O_{\mathfrak p}[\zeta_{\exp(G)}]$ is clean, we have $\ord_{\exp(G)}p=\varphi(\exp(G))/2$. Next we check all the cases of  Proposition \ref{p2.6}.2.

 Suppose $\exp(G)=2^r$ with $r\ge 2$. Then $|\Delta|=8$. Since $b$ is a maximal divisor such that $\Delta\nmid b$, we have  $b=4$ and hence $p\equiv 1\pmod 4$. Therefore  $r\ge 3$ and   $d=\pm2$.  If $G\not\cong C_{2^r}$ or $\Phi_4(x)=x^2+1$ is not irreducible over $\mathcal O_{\mathfrak p}$, then $\lambda(4)\ge 2$ and hence $\mathcal O_{\mathfrak p}[G]$ is not weakly clean, a contradiction. Thus $G\cong C_{2^r}$ and $\Phi_4(x)=x^2+1$ is  irreducible over $\mathcal O_{\mathfrak p}$. Since $\mathcal O_{\mathfrak p}[\zeta_8]$ is clean, we have $\varphi(8)/2=\ord_8p$ and hence $p\not\equiv 1\pmod 8$. Note that $p\equiv 1\pmod 4$. We have $p\equiv 5\pmod 8$. This is the case (i).

 Suppose $\exp(G)=q^r$ or $2q^r$ with $q$ is an odd prime and $\ord_{q^{\epsilon(r)}}p=q^{\epsilon(r)-1}(q-1)/2$. Then $b=q$ and   $q^2 \mid |\Delta|$. But this is impossible, since $d$ is square-free.

%

 Suppose $\exp(G)=4q^r$, where $q$ is an odd prime. Then $b=4$ or $q$ or $4q$. If $b=4q$, then $q^2\mid |\Delta|$. But this is impossible, since $d$ is square-free. If $b=q$, since $d$ is square-free, we obtain $|\Delta|=4q$. But $2q$ is a divisor of $\exp(G)$  such that $\Delta\nmid 2q$, a contradiction.
 Thus $b=4$, which implies  $|\Delta|=q$ or $4q$ and $p\equiv 1\pmod 4$. If $|\Delta|=q$,  since  $\mathcal O_{\mathfrak p}[\zeta_{4q}]$ and $\mathcal O_{\mathfrak p}[\zeta_{q}]$ are clean, we obtain $\varphi(4q)/2=\ord_{4q}p=\ord_qp$ and $\varphi(q)/2=\ord_qp$, a contradiction. Thus $|\Delta|=4q$ and $d=\pm q\equiv 1\pmod 4$. Since  $\mathcal O_{\mathfrak p}[\zeta_{q^r}]$ is clean, we obtain  $\varphi(q^r)=\ord_{q^r}p$, whence  $p$ is a primitive root of $q^{\epsilon(r)}$. If the Sylow-$2$ subgroup of $G$ is not isomorphic to $C_4$ or $\Phi_4(x)=x^2+1$ is not irreducible over $\mathcal O_{\mathfrak p}$, then $\lambda(4)\ge 4$ and hence  $\mathcal O_{\mathfrak p}[G]$ is not weakly clean, a contradiction. Thus $G\cong C_{4}\oplus H$, where $H$ is a $q$-group, and $\Phi_{4}(x)=x^2+1$ is  irreducible over $\mathcal O_{\mathfrak p}$. This is the case (ii).

 Suppose $\exp(G)=q_1^{r_1}q_2^{r_2}$ or $2q_1^{r_1}q_2^{r_2}$. Then by symmetry we may assume $b=q_2$ or $q_1q_2$. If $b=q_1q_2$, then $|\Delta|$ is not square-free, a contradiction. Thus $b=q_2$, $\ord_{q_2}p=(q_2-1)/2$, and $|\Delta|=q_1$ or $q_1q_2$.
 Since $\mathcal O_{\mathfrak p}[\zeta_{q_1q_2}]$ is clean, we have $\varphi(q_1q_2)/2=\ord_{q_1q_2}p=\lcm\{(q_2-1)/2, \ord_{q_1}p\}$, which implies that $\ord_{q_1}p=\varphi(q_1)$ and hence $|\Delta|\nmid q_1$. Therefore $|\Delta|=q_1q_2$ and $d=\pm q_1q_2\equiv 1\pmod 4$. In view of Proposition \ref{p2.6}.2.(e), we obtain that $p$ is a primitive root of $q_1^{\epsilon(r_1)}$ and $\gcd(q_1^{\epsilon(r_1)-1}(q_1-1), (q_2-1))=2$. If the Sylow-$q_2$ subgroup of $G$ is not isomorphic to $C_{q_2}$ or $\Phi_{q_2}(x)$ is not irreducible over $\mathcal O_{\mathfrak p}$, then $\lambda(q_2)\ge 2$ and hence $\mathcal O_{\mathfrak p}[G]$ is not weakly clean, a contradiction. Hence $G\cong C_{q_2}\oplus H$, where $H$ is a $p_1$-group, and $\Phi_{q_2}(x)$ is  irreducible over $\mathcal O_{\mathfrak p}$. This is the case (iii).

\smallskip
\noindent{\bf Case 2. }$\Delta\mid b$.
\smallskip

Since $\mathcal O_{\mathfrak p}[\zeta_{b}]$ is not clean implies that
 $\mathcal O_{\mathfrak p}[\zeta_{\exp(G)}]$ is  not clean, it follows from Lemma \ref{lemm2} and Theorem \ref{t1} that  $b=\exp(G)$.
 Therefore $\varphi(\exp(G))/2=2\ord_{\exp(G)}p$. It is easy to see that $G\cong C_{\exp(G}$ and $\Phi_{\exp(G)}(x)$ is irreducible over $\mathcal O_{\mathfrak p}$.

Suppose $|\Delta|=\exp(G)$. Then $\varphi(|\Delta|)/2=2\ord_{|\Delta|}p$ and
 for all divisors $m$ of $\Delta$ with $m\neq |\Delta|$, we have $\mathcal O_{\mathfrak p}[\zeta_m]$ is clean and hence $\ord_mp=\varphi(m)$. In view of Proposition \ref{p3.2}.5, one of the following holds.

 	\begin{itemize}
 	
 	\item $\exp(G)=|\Delta|=d=q\equiv 1\pmod 4$, where $q$ is an odd prime with $\ord_{q}p=(q-1)/4$. This is the first case of (iv).
 	
 	\item $\exp(G)=|\Delta|=q_1q_2$, $q_1\equiv 1\pmod 4$, $q_2\equiv 1\pmod 4$, and $\Delta=d=q_1q_2$,  where $q_1,q_2$ are distinct  odd primes,  and  $p$ is a primitive root of both $q_1, q_2$  and $\gcd(q_1-1, q_2-1)=4$. This is the second case of (iv).
 	
 \end{itemize}

Suppose $|\Delta|\neq \exp(G)$. Then $\varphi(|\Delta|)/2=\ord_{|\Delta|}p$ and
for all divisors $m$ of $\Delta$ with $m\neq |\Delta|$, we have $\mathcal O_{\mathfrak p}[\zeta_m]$ is clean and hence $\ord_mp=\varphi(m)$.
By Proposition \ref{p3.2}.2, one of the following holds.

\begin{itemize}
	\item $|\Delta|=8$ and $p\equiv 3\pmod 4$.

	\item $|\Delta|=q$ is an odd prime  and  $\ord_{q}p=(q-1)/2$.

	\item $|\Delta|=4q$, where $q$ is an odd prime,  $p\equiv 3\pmod 4$, and    $p$ is a primitive root of $q$.

	\item $|\Delta|=q_1q_2$ is a product of  two distinct primes, $p$ is a primitive root of unity of both $q_1$ and $q_2$, and $\gcd(q_1-1,q_2-1)=2$.
\end{itemize}

 Let $m$ be a divisor of $\exp(G)$ with $m\neq \exp(G)$. If $\Delta\nmid m$, then $\ord_mp=\varphi(m)$, i.e., $p$ is a primitive root of $m$.  If $\Delta\t m$, then $\ord_mp=\varphi(m)/2$. We will use those  facts without further mention. Note that $\varphi(\exp(G))=4\ord_{\exp(G)}p$.

If $|\Delta|=8$, then by Proposition \ref{p2.6}.3 we obtain either $\exp(G)=16$ with $p\equiv 7, 15\pmod {16}$, or $\exp(G)=8q^r$ with $p\equiv 3\pmod 4$, and $p$ is a primitive root of $q^r$. If $\exp(G)=8q^r$, since $8\nmid 4q^r$ and $4q^r$ has no primitive root, we get a contradiction.
 Thus $\exp(G)=16$ and $d=\pm 2$. This is the third case of (iv).

If $|\Delta|=4q$, where  $q$ is an odd prime  such that $p$ is a primitive root of $q$, then by Proposition \ref{p2.6}.3 and $8, 4q'$ have no primitive root, where $q'$ is an odd prime, we obtain that this is impossible.

If $|\Delta|=q_1q_2$ is a product of  two distinct primes, $p$ is a primitive root of unity of both $q_1$ and $q_2$, and $\gcd(q_1-1,q_2-1)=2$,  then by Proposition \ref{p2.6}.3 and $q_1'q_2', 4q_1'$ have no primitive root, where $q_1', q_2'$ are two distinct odd primes, we obtain that this is impossible.

If $|\Delta|=q$ is an odd prime with $\ord_qp=(q-1)/2$, then by Proposition \ref{p2.6}.3 and $8, q_1'q_2', 4q_1'$ have no primitive root, where $q_1', q_2'$ are two distinct odd primes, we obtain that one of the following holds.

\begin{itemize}
	\item $\exp(G)=qq_2$, where $q_2$ is a prime such that $p$ is a primitive root of $q_2$ and $\gcd(q_2-1, (q-1)/2)=2$. This is the fourth case of (iv).
	
	\item $\exp(G)=4q$ with $p\equiv 3\pmod 4$ and  $q\equiv 1\pmod 4$. This is the fifth case of (iv).
	
\end{itemize}

\smallskip
\textbf{(3)}
  Suppose $\Delta\mid \exp(G)$ and  $\big(\frac{\Delta}{p}\big)= -1$. Then $N(\mathsf p)=p^2$.
By Lemma \ref{3.4} we have $\Q(\sqrt{d})\subset \Q(\zeta_{\exp(G)})$. Let $m$ be a divisor of $\exp(G)$ with $\Delta\nmid m$.
Then $\deg(\phi_m(x))=[\Q(\sqrt{d})(\zeta_m):\Q(\sqrt{d})]=\varphi(m)$. Let $m$ be a divisor of $\exp(G)$ with $\Delta\mid m$.
Then $\deg(\phi_m(x))=[\Q(\sqrt{d})(\zeta_m):\Q(\sqrt{d})]=\varphi(m)/2$.

If $4$ divides $\exp(G)$, then $\Delta\nmid 4$ and $\deg(\phi_4(x))=\varphi(4)=2\ord_4p^2$ and hence $\mathcal O_{\mathfrak p}[\zeta_4]$ is not clean.
Let $m$ be a divisor of $\exp(G)$ such that $\Delta\nmid m$ and $\mathcal O_{\mathfrak p}[\zeta_{m}]$ is clean. Then $\deg(\phi_m(x))=\varphi(m)=\ord_mp^2$, which implies that $\ord_mp=\varphi(m)$ is odd.
Thus $m=1$ or $2$. We will use those facts without further mention.

\smallskip
\textbf{(3.a)} It is clear that
 $\mathcal O_{\mathfrak p}[G]$ is clean if and only if  for every divisor $m$ of $\exp(G)$, we have $\mathcal O_{\mathfrak p}[\zeta_{m}]$ is clean.

 Suppose  $\mathcal O_{\mathfrak p}[G]$ is clean. Let $m$ be a divisor of $\exp(G)$.
 If $\Delta\nmid m$, then $m=1$ or $2$. Therefore $|\Delta|=q$ is a prime and $d=\pm q\equiv 1\pmod 4$. Suppose $\exp(G)=q^rm'$ with $\gcd(q,m')=1$. Then $m'=1$ or $2$. Therefore $\exp(G)=q^r$ or $2q^r$. Since  $\mathcal O_{\mathfrak p}[\zeta_{q^r}]$ is clean, we have $\varphi(q^r)/2=\ord_{q^r}p^2$, whence  either $p$ is a primitive root of $q^{\epsilon(r)}$, or $q\equiv 3\pmod 4$ and $\ord_{q^{\epsilon(r)}}p=q^{\epsilon(r)-1}(q-1)/2$.

The converse  follows directly from a careful check.

\smallskip
\textbf{(3.b)} It is clear that
$\mathcal O_{\mathfrak p}[G]$ is feebly clean if and only if  for every divisor $m$ of $\exp(G)$, we have $\mathcal O_{\mathfrak p}[\zeta_{m}]$ is feebly clean.

Suppose $\mathcal O_{\mathfrak p}[G]$ is feebly clean. Then $\mathcal O_{\mathfrak p}[\zeta_{\exp(G)}]$ is feebly clean, which implies that $\ord_{\exp(G)}p^2=\varphi(\exp(G))/2$ or $\varphi(\exp(G))/4$. We distinguish two cases.

\smallskip
\noindent{\bf Case 1. } $\ord_{\exp(G)}p^2=\varphi(\exp(G))/2$, i.e. $\ord_{\exp(G)}p=\varphi(\exp(G))$ with $\exp(G)\ge 3$, or
$\ord_{\exp(G)}p=\varphi(\exp(G))/2$ is odd.
\smallskip

In this case we only need to show that $\mathcal O_{\mathfrak p}[G]$ is feebly clean and this follows from the fact that $\ord_{m}p^2\ge \varphi(m)/2$ for all divisors $m$ of $\exp(G)$.

\smallskip
\noindent{\bf Case 2. } $\ord_{\exp(G)}p^2=\varphi(\exp(G))/4$, i.e. $\ord_{\exp(G)}p=\varphi(\exp(G))/2$ is even, or
$\ord_{\exp(G)}p=\varphi(\exp(G))/4$ is odd.
\smallskip

If one of the cases (i)-(x) holds, then a careful check implies that $\mathcal O_{\mathfrak p}[\zeta_{m}]$ is feebly clean for all divisors $m$ of $\exp(G)$ and hence $\mathcal O_{\mathfrak p}[G]$ is feebly clean. We only need to show that $\mathcal O_{\mathfrak p}[G]$  feebly clean implies that one of the cases of (i)-(x) holds. Now we suppose $\mathcal O_{\mathfrak p}[G]$ is feebly clean and we distinguish two cases.

\smallskip
\noindent{\bf Subcase 2.1. } $\ord_{\exp(G)}p=\varphi(\exp(G))/2$ is even.
\smallskip

 We now check all the cases of Proposition \ref{p2.6}.2.

If $\exp(G)=4$, then $\varphi(4)/2$ is not even. If $\exp(G)=8$, then $|\Delta|=8$, $p\not\equiv 1\pmod 8$, and $d=\pm 2$. This is the case (i) with $p\not\equiv 1\pmod 8$.

If $\exp(G)=2^r$ with $r\ge 4$ and $p\equiv \pm3,\pm5\pmod {16}$, then  $|\Delta|=8$ and $d=\pm 2$. This is the case (ii).

If $\exp(G)=q^r$ or $2q^r$ with $\ord_{q^{\epsilon}}p=q^{\epsilon(r)-1}(q-1)/2$ is even, then $q\equiv 1\pmod 4$ and  $\Delta=q=d$. This is the case (iii).

If $\exp(G)=4q^r$ with either $p$ is a primitive root of $q^{\epsilon(r)}$, or $p\equiv 3\pmod 4$, $q\equiv 3\pmod 4$, and $\ord_{q^{\epsilon(r)}}p=q^{\epsilon(r)-1}(q-1)/2$, then $|\Delta|=q$ or $4q$. If $|\Delta|=q$, then $d=\pm q\equiv 1\pmod 4$ and this is 3.b.(v). If $|\Delta|=4q$, then $d=\pm q\equiv 3\pmod 4$ and this is the case (vi).

  If  $\exp(G)=q_1^{r_1}q_2^{r_2}$ or $2q_1^{r_1}q_2^{r_2}$, where $q_1,q_2$ are distinct odd primes with $q_2\equiv 3\pmod 4$ and $\gcd(q_1^{\epsilon(r_1)-1}(q_1-1), q_2^{\epsilon(r_2)-1}(q_2-1))=2$,  and either $p$ is a primitive root of both $q_1^{\epsilon(r_1)},q_2^{\epsilon(r_2)}$, or $p$ is a primitive root of $q_1^{\epsilon(r_1)}$ and $\ord_{q_2^{\epsilon(r_2)}}p=q_2^{\epsilon(r_2)-1}(q_2-1)/2$, then $|\Delta|=|d|\in \{q_1, q_2, q_1q_2\}$ and $d\equiv 1\pmod 4$. This is the case (ix).

\smallskip
\noindent{\bf Subcase 2.2. } $\ord_{\exp(G)}p=\varphi(\exp(G))/4$ is odd.
\smallskip

We check all the cases of Proposition \ref{p3.2}.3.

 If $\exp(G)=8$ and $p\equiv 1\pmod 8$, then $|\Delta|=8$ and $d=\pm 2$. This is the case (i) with $p\equiv 1\pmod 8$.

If $\exp(G)=q^r$ or $2q^r$ with $\ord_{q^{\epsilon}}p=q^{\epsilon(r)-1}(q-1)/4$ and $q\equiv 5\pmod 8$, then $\Delta=q=d$. This is the case (iv).

If $\exp(G)=q_1^{r_1}q_2^{r_2}$ or $2q_1^{r_1}q_2^{r_2}$, where $q_1,q_2$ are distinct  odd primes with $q_1\equiv 3\pmod 4$ and $q_2\equiv 3\pmod 4$,  and either $p$ is a primitive root of $q_1^{\epsilon(r_1)}$, $\ord_{q_2^{r_2}}p=q_2^{r_2-1}(q_2-1)/2$, and $\gcd(q_1^{r_1-1}(q_1-1), q_2^{r_2-1}(q_2-1)/2)=1$,  or $\ord_{q_1^{r_1}}p=q_1^{r_1-1}(q_1-1)/2$, $\ord_{q_2^{r_2}}p=q_2^{r_2-1}(q_2-1)/2$, and $\gcd(q_1^{r_1-1}(q_1-1)/2, q_2^{r_2-1}(q_2-1)/2)=1$, then $|\Delta|=|d|\in \{q_1,q_2,q_1q_2\}$ and $d\equiv 1\pmod 4$. This is the case (x).

If  $\exp(G)=4q^r$, where $q$ is a prime with $q\equiv 3\pmod 4$,
$p\equiv 1\pmod 4$, and $\ord_{q^r}p=q^{r-1}(q-1)/2$, then $|\Delta|=q$ or $4q$. If $|\Delta|=q$, then $d=-q$ and this is 3.b.(vii). If $|\Delta|=4q$, then $d=q$ and this is the case (viii).

%

\smallskip
\textbf{(3.c)} It is easy to check that  if
$G\cong C_q$, $\Phi_q(x)$ is irreducible over $\mathcal O_{\mathfrak p}$,
$|\Delta|=q$,   $d= q\equiv 1\pmod 4$,  and either $\ord_qp=(q-1)/4$ is odd, or $\ord_qp=(q-1)/2$, where $q$ is an odd prime, then
 $\mathcal O_{\mathfrak p}[G]$ is weakly clean but not clean. Now we prove the converse and
suppose $\mathcal O_{\mathfrak p}[G]$ is weakly clean but not clean.
Let $b$ be the only divisor of $\exp(G)$ such that $\mathcal O_{\mathfrak p}[\zeta_{m}]$ is not clean. Then $b\ge 3$. We distinguish two cases.

\smallskip
\noindent{\bf Case 1. } $\Delta\nmid b$.
\smallskip

If there exists a divisor $b'$ of $\exp(G)$ such that $\Delta\nmid b'$ and  $b\t b'$ with $b\neq b'$, then
$\mathcal O_{\mathfrak p}[\zeta_{b'}]$ is  clean and hence $\mathcal O_{\mathfrak p}[\zeta_{b}]$ is  clean, a contradiction. Thus
$b$ is  a maximal divisor of $\exp(G)$ such that $\Delta\nmid b$.

 For every divisor $b_1$ of $b$ with $b_1\neq b$, we have $\mathcal O_{\mathfrak p}[\zeta_{b_1}]$ is clean. Note that $\Delta\nmid b_1$. We have $b_1=1$ or $2$.
 Therefore $b=4$ or $b$ is an odd prime.

   Note that $\mathcal O_{\mathfrak p}[\zeta_{\exp(G)}]$ is  clean and hence $\ord_{\exp(G)}p^2=\varphi(\exp(G))/2$, which  implies that $\ord_{\exp(G)}p=\varphi(\exp(G))/2$ is odd or $\ord_{\exp(G)}p=\varphi(\exp(G))$ with $\exp(G)\ge 3$. Since $b\neq \exp(G)$, it follows from Proposition \ref{p3.2}.1 and Proposition \ref{p2.6}.1  that $b$ is an odd prime and $\exp(G)=b^r$ or $2b^r$ for some $r\in \N$. The maximality of $b$ implies that $b^2\mid \Delta$, a contradiction to the fact that $d$ is square-free.

%

%

 \smallskip
 \noindent{\bf Case 2. } $\Delta\mid b$.
 \smallskip

 Since $\mathcal O_{\mathfrak p}[\zeta_{b}]$  not clean implies that
 $\mathcal O_{\mathfrak p}[\zeta_{\exp(G)}]$ is  not clean, it follows from Lemma \ref{lemm2} and Theorem \ref{t1} that  $b=\exp(G)$.

For all divisors $m$ of $|\Delta|$ with $m\neq |\Delta|$, we have $\Delta\nmid m$ and $\mathcal O_{\mathfrak p}[\zeta_{m}]$ is clean, which implies $m=1$ or $2$.  Thus $|\Delta|=q$ is a prime and $\exp(G)=q^rm'$, where $m'\in \N$ with $\gcd(q, m')=1$. Then $m'=1$ or $2$, whence $\exp(G)=q^r$ or $2q^r$ for some $r\in \N$.

 If $\exp(G)=2q^r$, then $\mathcal O_{\mathfrak p}[\zeta_{q^r}]$ is clean and hence $\mathcal O_{\mathfrak p}[\zeta_{\exp(G)}]$ is  clean, a contradiction. Thus $\exp(G)=q^r$. If $r\ge 2$, then $\mathcal O_{\mathfrak p}[\zeta_{q}]$ is clean and hence $\ord_qp^2=\varphi(q)/2$.
 Therefore $\ord_{q^2}p^2=\varphi(q)/2$ or $q\varphi(q)/2$. If $\ord_{q^2}p^2=q\varphi(q)/2$, then $\ord_{q^r}p^2=q^{r-1}\varphi(q)/2$ and hence $\mathcal O_{\mathfrak p}[\zeta_{\exp(G)}]$ is  clean, a contradiction. If $\ord_{q^2}p^2=\varphi(q)/2$, then $\ord_{q^2}p^2=\varphi(q^2)/2q<\varphi(q^2)/4$, which implies that $\mathcal O_{\mathfrak p}[\zeta_{p^2}]$ is not feebly clean, a contradiction. Therefore $r=1$ and $\ord_qp^2=\varphi(q)/4$, which implies that either $\ord_qp=(q-1)/4$ is odd, or $\ord_qp=(q-1)/2$ is even. Both imply $q\equiv 1\pmod 4$ and hence $d=q=\Delta$.

 If  $G\not\cong C_{q}$ or $\Phi_q(x)$ is not irreducible over $\mathcal O_{\mathfrak p}$, then $\lambda(q)\ge 2$ and hence $\mathcal O_{\mathfrak p}[G]$ is not weakly clean, a contradiction. Thus $G\cong C_{q}$ and  $\Phi_q(x)$ is  irreducible over $\mathcal O_{\mathfrak p}$.
\end{proof}

%
%
%

%
%
%


\begin{thebibliography}{1}
\bibitem{Ah-An06} Myung-Sook ~Ahn and D.D. ~Anderson, \emph{Weakly clean rings and almost clean rings}, Rocky Mountain J. Math. \textbf{3}(2016), 783 -- 789.



\bibitem{An-Ca02}D.D. ~Anderson and V.P. ~Camillo, \emph{Commutative rings whose elements are a sum of a unit and idempotent}, Commun. Algrbra \textbf{7}(2002), 3327 -- 3336.

\bibitem{Ar-Ku17}N.~Arora and S. ~Kundu, \emph{Commutative feebly clean rings}, J. Algebra and its Applications \textbf{5}(2017), 1750128, 14p.




\bibitem{Ca-Kh-La-Ni-Zh06}V.P. ~Camillo, D. ~Khurana, T.Y. ~ Lam. W.K. Nicholson, and Y. Zhou, \emph{Continuous modules are clean}, J. Algebra \textbf{304}(2016), 94 -- 111.

\bibitem{Ca-Yu94}V.P. Camillo and Hua-Ping Yu, \emph{Exchange rings, units and idempotents}, Communications in Algebra, \textbf{22}(1994), 4737 -- 4749.


\bibitem{Ch-Zh07}J. Chen and Y. Zhou, \emph{Strongly clean power series rings}, Proc. Edinb. Math. Soc.(2), \textbf{50}(2007), 73 -- 85.


 \bibitem{Ch-Qu11}A. T. M. ~Chin and K. T. ~Qua, \emph{A note on weakly clean rings}, Acta Math. Hungar. \textbf{1-2}(2011), 113 -- 116.

\bibitem{Da14} P. V. ~Danchev, \emph{On weakly exchange rings}, J. Math. Univ. Tokushima \text{48}(2014), 17 –- 22.

\bibitem{Fr-Ta92}
A.~Fr\"{o}hlich and M.~J.~Taylor , \emph{Algebraic number theory}, cup ed., Cambridge
  Studies in Advanced Mathematics, Cambridge University Press, 1992.



\bibitem{Ha-Ni01}
J. Han and W.K. Nicholson,  \emph{Extensions of clean rings}, Comm. Algebra \textbf{29} (2001) 2589-255.




\bibitem{Ko-Sa-Zh17}T. Kosan,  S. Sahinkaya, and Y. Zhou, \emph{On weakly clean rings}, Commun. Algebra \textbf{8}(2017), 3494 -- 3502.


\bibitem{Im-Mc14a}
Nicholas~A. Immormino and Warren~Wm. McGovern, \emph{Examples of clean
  commutative group rings}, J. Algebra \textbf{405} (2014) 168 -- 178.



\bibitem{Li-Zh20}Y. Li and Q. Zhong, \emph{Clean group rings over localizations of rings of integers},  J. Pure Appl. Algenra  \textbf{224}(2020)106284.


\bibitem{Mc06a}
W. Wm. McGovern, \emph{A characterization of commutative clean rings}, Int. J. Math. Game Theory Algebra \textbf{15}(2006), 403 -- 413.


\bibitem{Mc06b}
W. Wm. McGovern, \emph{Neat rings}, J. Pure Appl. Algebra \textbf{205}(2006), 243 -- 265.


\bibitem{Mc18}W. Wm. McGovern, \emph{The group ring $\Z_{(p)}C_q$ and Ye's theorem}, J. Algebra Appl. \textbf{17} (2018), 1850111(5 pages).

\bibitem{Mo-Na06} Todor Zh. Mollov and Nako A. Nachev, \emph{Unit groups of commutative group rings}, Commun. Algebra \textbf{34}(2006), 3835 -- 3857.



\bibitem{Ni77}
W.K. Nicholson, \emph{Lifting idempotens and exchange rings}, Trans. Amer. Math. Soc. \textbf{220} (1977) 269-278.


\bibitem{We06}
Steven~H. Weintraub, \emph{Galois theory}, 1 ed., Universitext, Springer, 2006.


\bibitem{Zh10}
Yiqiang Zhou, \emph{On clean group rings}, Advances in ring theory, 335-345, Trends Math., Birkhauser/Springer Basel AG, Basel, 2010.
\end{thebibliography}

\providecommand{\bysame}{\leavevmode\hbox to3em{\hrulefill}\thinspace}
\providecommand{\MR}{\relax\ifhmode\unskip\space\fi MR }
\providecommand{\MRhref}[2]{%
  \href{http://www.ams.org/mathscinet-getitem?mr=#1}{#2}
}
\providecommand{\href}[2]{#2}

\end{document}